\documentclass{amsart}
\usepackage{enumerate}

\newtheorem{Thm}{Theorem}[section]
\newtheorem{Lem}[Thm]{Lemma}
\newtheorem{Def}[Thm]{Definition}
\newtheorem{Prop}[Thm]{Proposition}
\newtheorem{Coro}[Thm]{Corollary}
\newtheorem{Rem}{Remark}[section]

\newcommand{\p}{\partial}
\renewcommand{\div}{\sum_{\alpha=1}^d\p_{\alpha}}
\newcommand{\suma}{\sum_{\alpha=1}^d}
\newcommand{\R}{\mathbb R}
\newcommand{\N}{\mathbb N}
\newcommand{\etab}{\eta}
\newcommand{\xib}{\xi}

\newcommand{\ninf}[1]{\| #1 \|_{\infty}}

\newcommand{\T}{\mathcal T}
\newcommand{\E}{(K,L) \in\mathcal E_r}
\newcommand{\dt}{\Delta t}
\newcommand{\NK}{\mathcal N(K)}
\newcommand{\intt}{\displaystyle\int_{t^n}^{t^{n+1}}}
\newcommand{\intRR}{\iint_{\R^d \times \R_+}}
\newcommand{\intR}{\int_{\R^d}}
\renewcommand{\O}{\Omega}
\newcommand{\sumk}{\sum_{n=0}^{N_T}}
\newcommand{\Tr}{\T_r}
\newcommand{\dTr}{\partial\T_r}
\newcommand{\sumNK}{\sum_{L\in \NK}}
\newcommand{\sumK}{\sum_{K\in\Tr}}
\newcommand{\sumE}{\sum_{\E}}
\newcommand{\mub}{\overline{\mu}}
\newcommand{\mubo}{\overline{\mu}_0}

\def\ov#1{\overline{#1}}

\def\Ee{\mathcal{E}}

\def\nico#1{{
#1}}

\def\be{\begin{equation}}
\def\ee{\end{equation}}
\def\a{\alpha}
\def\sKL{{\sigma_{KL}}}
\def\Mm{\mathcal{M}}
\def\Cc{\mathcal{C}}
\def\Tt{\mathcal{T}}
\def\Ii{\mathcal{I}}

\def\nn{\nonumber}

\def\oO{\overline{\Omega}}

\begin{document}

\title[Error estimate for explicit FV scheme]
{Error estimate for time-explicit finite volume approximation of strong solutions to systems of conservation laws}

\author{Cl{\'e}ment Canc{\`e}s}
\address{
Cl\'ement Canc\`es ({\tt clement.cances@inria.fr}).
Team RAPSODI, Inria Lille -- Nord Europe, 40 av. Halley, F-59650 Villeneuve d'Ascq, France.
}

\thanks{This work was supported by the LRC Manon (Mod{\'e}li\-sation et approximation
num{\'e}rique orient{\'e}es pour l'{\'e}nergie nucl{\'e}aire -- CEA/DM2S-LJLL) 
and by the Nuclear System and Scenarios federative project of the NEEDS
program (CNRS -- CEA -- AREVA -- EDF -- IRSN)}
\author{H{\'e}l{\`e}ne Mathis
}
\address{
H\'el\`ene Mathis ({\tt helene.mathis@univ-nantes.fr}).
Universit{\'e} de Nantes, Laboratoire de Math{\'e}matiques Jean Leray, 2, Rue
de la Houssini{\`e}re, 44322 Nantes Cedex 03, France.
}
\author{Nicolas Seguin}
\address{
Nicolas Seguin ({\tt nicolas.seguin@univ-nantes.fr}).
Universit{\'e} de Nantes, Laboratoire de Math{\'e}matiques Jean Leray, 2, Rue
de la Houssini{\`e}re, 44322 Nantes Cedex 03, France.
}

\maketitle


%



\begin{abstract}
  We study the finite volume approximation of strong solutions to
  nonlinear systems of conservation laws.  We focus on time-explicit
  schemes on unstructured meshes, with entropy satisfying numerical
  fluxes.  The numerical entropy dissipation is quantified at each
  interface of the mesh, which enables to prove a weak--$BV$ estimate
  for the numerical approximation under a strengthened CFL condition.
  Then we derive error estimates in the multidimensional case, using
  the relative entropy between the strong solution and its finite
  volume approximation. The error terms are carefully studied, leading
  to a classical $\mathcal{O}(h^{1/4})$ estimate in $L^2$ under this
  strengthened CFL condition.
\end{abstract}
\vspace{10pt}

{\small {\bf Keywords.}
Hyperbolic systems, finite volume scheme, relative entropy, error estimate
\vspace{0pt}

{\bf AMS subjects classification. }
35L65, 65M08, 65M12, 65M15
}


\section{Introduction}

The aim of this paper is to provide an {\em a priori} 
error estimate for time-explicit finite volume approximation 
on unstructured meshes of strong solutions to hyperbolic systems 
of conservation laws. Our proof relies on the control of 
perturbations coming from the discretization in the uniqueness 
proof proposed by R. J. DiPerna~\cite{DiP79} and C. M. Dafermos~\cite{Daf79} 
(see also~\cite{dafermosBook}). 

Numerous studies on error estimates for hyperbolic problems were 
published in the last decades. Let us first highlight some optimal 
convergence rates that are established in the literature. 
Classical first-order finite 
  difference methods on cartesian grids for the
  approximation of smooth solutions of linear equations can be
  directly studied by estimating the truncation error, leading to 
  an $\mathcal{O}(h)$ error estimate, where the length $h$ is the characteristic 
  size of the grid.
  Adapting S. N. Kruzhkov's doubling variable
  technique \cite{Kru70}, N. N. Kuznetsov proved in \cite{Kuz76} that finite
  difference schemes for nonlinear one-dimensional conservation laws
  converge towards the entropy weak solution with the optimal rate
  $\mathcal{O}(h^{1/2})$ in the space-time $L^1$ norm. 
  The optimal rate $\mathcal{O}(h^{1/2})$ has been recovered by 
  B. Merlet and J. Vovelle~\cite{MV07} 
  and by F. Delarue and F. Lagouti\`ere~\cite{DL11} 
  for weak solutions to the linear transport equation approximated by 
  the upwind finite volume scheme on unstructured grids.
  The rate $\mathcal{O}(h^{1/2})$ appears to stay optimal when 
  strong solutions to linear transport equations are approximated on two-dimensional 
  unstructured grids as shown by C. Johnson and J. Pikt\"aranta in~\cite{JP}.
  Since linear transport enters our framework, we cannot expect a better {\em a priori} error estimate 
  than $\mathcal{O}(h^{1/2})$. However, if one restricts to 
  dimensional one, an estimate in $\mathcal{O}(h)$ can be derived even for nonlinear
  systems of conservation laws, see D. Bouche \emph{et al.}  \cite{BGP}.

Many studies exist when considering entropy
  \emph{weak} solutions, based on nonlinear techniques which extend in
  some sense N. N. Kuznetsov's article \cite{Kuz76}. These works focus
  on the multidimensional case with unstructured meshes, for scalar
  conservation laws \cite{CGH93, CClF94,vila94,EGGH98,chainais99}.
  They mainly use the notion of error measures, see for instance
  \cite{BP98}, and lead the an error estimate in
  $\mathcal{O}(h^{1/4})$ (recall that the convergence has been
  initially addressed by A. Szepessy~\cite{SzCV}). The
  key-point in these multidimensional studies is the control of the
  $BV$ semi-norm. For unstructured meshes, one can only prove that it
  grows as $h^{-1/2}$ (even in the linear scalar case,
  cf.~\cite{Des04}), which actually is the main barrier to obtain a
  better rate of convergence.
  Similar tools allowed V. Jovanovic and C. Rohde to propose 
  error estimates~\cite{JR05} for the finite volume approximation of the solution 
  to Friedrich's systems (i.e., linear symmetric hyperbolic systems)

As mentioned above, we are interested in multidimensional
  systems of conservation laws.  The solutions to such systems may
  develop discontinuities in finite time and, since the pioneering
  work of P. D. Lax \cite{Lax2}, entropy conditions are added to select
  physical/admissible solutions. Recently, it has been shown by C. De
  Lellis and L. Sz{\'e}kelyhidi Jr. in \cite{DLS09,DLS10} that such a
  criterion is not sufficient in the multidimensional case.
  Nonetheless, it is known since several decades (see in
  particular~\cite{DiP79,Daf79}) that if a strong solution exists,
  then there exists a unique entropy weak solution corresponding to
  the same initial data, and that it coincides with this strong
  solution.  Moreover, it can be shown that entropy weak solutions are
  stable with respect to strong solutions. Since error estimates of
  any approximation are based on the stability properties of the
  model, we restrict this study to strong solutions which are known to
  exist, in finite time, and to be unique
  \cite{Kato75,LiTaTsien_book,dafermosBook}.   
  As in \cite{JR06}, we use
  the notion of relative entropy to compare the approximate solution
  with a smooth solution. The mathematical techniques are basically
  the same as in the scalar case (we follow in particular
  \cite{EGGH98, chainais99}): weak-$BV$ estimates and error measures.  
  The main result of this paper is an {\em a priori} error estimate of 
  order $\mathcal{O}(h^{1/4})$ in the space-time $L^2$ norm for 
  first-order time-explicit finite volume schemes under classical assumptions on
  the numerical fluxes \cite{bouchut_book,TadmorBook}. 
  One key ingredient is an extension to the system case of the so-called 
  weak--$BV$ estimate introduced in the scalar case in \cite{EGGH98, chainais99}.
  As in the scalar, it relies on a quantification of the numerical dissipation 
  and requires a slightly reduced CFL condition. 
  We finally obtain an error estimate in $\mathcal{O}(h^{1/4})$, and
  simplify the framework of a study of V. Jovanovic and C. Rohde~\cite{JR06} 
  where time-implicit methods are considered and the weak-$BV$ estimate is assumed). 

Concerning higher order methods, let us mention the result~\cite{CH00} of 
C. Chainais-Hillairet who proved an error estimate or order $\mathcal{O}(h^{1/4})$ 
for the time-explicit second order finite volume discretization with flux limiters~\cite{vL79}
of nonlinear scalar conservations laws. 
The strategy exploited in~\cite{CH00} consists in showing that the solution to the second order 
remains close to the solution to the monotone scheme without limiters. Such a strategy 
might be adapted in our framework but provides very under-optimal estimates. 
	High order time-implicit discontinuous Galerkin
    methods have also been analyzed in details by Hiltebrand and Mishra in
    \cite{HM14}. Using appropriate weak-$BV$ estimates, they prove
    convergence towards entropy measured-valued solutions of
    multidimensional systems of conservation laws. No error estimate
    has been derived yet up to our knowledge.

\begin{Rem}
As it is well known, solutions to hyperbolic nonlinear systems of conservation 
laws may develop discontinuities after a finite time. However, the occurence 
of such discontinuities is prohibited when appropriate relaxation terms are added 
to the systems~\cite{HN03, Yong04} (we then have hyperbolic balance laws instead of hyperbolic 
conservation laws). 
By adapting the analysis carried out by V. Jovanovic and C. Rohde in~\cite{JR06}, 
such terms can be considered in the analysis. A time explicit treatment of the source terms 
would lead to a reduced CFL. Therefore, in the case of stiff relaxation terms, an implicit 
treatment of the source terms is relevant. We refer to the work~\cite{chainais01} of C. Chainais-Hillairet 
and S. Champier for an error estimate in the case of a scalar balance law. 
\end{Rem}

\subsection{Hyperbolic systems of conservation laws}

\subsubsection{Strong, weak, and entropy weak solutions}
We consider a  system of $m$ conservation laws 
\begin{equation}\label{eq:P}
\p_t u(x,t) + \sum_{\a =1}^d \p_\a f_\a(u)(x,t) = 0. 
\end{equation}
System~\eqref{eq:P} is set on the whole space $x \in \R^d$, and for
any time $t\in[0,T]$, $T>0$.  We assume that there exists a convex
bounded subset of $\R^m$, denoted by $\O$ and called \emph{set of the
  admissible states} such that
\begin{equation}\label{eq:Omega}
u(x,t) \in \Omega, \quad \forall (x,t) \in \R^d \times [0,T]. 
\end{equation}
System~\eqref{eq:P} is complemented with the initial condition
\begin{equation}\label{eq:u0}
u(x,0) = u_0(x) \in \O, \quad \forall x \in \R^d. 
\end{equation}

We assume for all $\a \in \{1,\dots, d\}$ the functions $f_\a : \R^m
\to \R^m$ to belong to 
$C^2(\oO;\R^m)$, and be such that $D f_\a$ are diagonalizable with 
real eigenvalues, where $D$ denotes the differential with respect to
the variables $u$.

System~\eqref{eq:P} is endowed with a uniformly convex entropy 
$\eta \in C^2(\oO; \R)$ such that there exists $\beta_1\ge \beta_0 >0$ so that 
\begin{equation}\label{eq:beta}
{\rm spec}\left( D^2 \eta(u) \right) \subset [\beta_0;\beta_1], \quad \forall
u \in \oO,
\end{equation} 
and the corresponding entropy flux $\xi \in C^2(\oO;\R^d)$ 
satisfies for all $\a \in \{1,\dots, d\}$
\be
\label{eq:xi}
D \xi_\a(u) =  D\eta (u) D f_\a(u),
  \qquad \forall u \in \O.
\end{equation}
Without loss of generality, we assume that 
$\eta(u) \ge 0$ for all $u \in \oO.$
The existence of the entropy flux $\xi$ amounts to assume the
integrability condition (see e.g.~\cite{GR96})
\begin{equation}\label{eq:entro_commute}
D^2 \eta(u)  D f_\a(u) = D f_\a(u)^T D^2\eta(u), \qquad \forall u \in
\O.
\end{equation}
Let us introduce the quantity $L_f$ by
\begin{equation}\label{eq:Lf_Rayleigh}
L_f = \sup_{\a \in \{1,\dots,d\}} \sup_{(u,v) \in \O^2} \sup_{w \in
  \R^m\setminus\{0\}} 
\left| \frac{w^T D^2\eta(v) D f_\alpha(u) w}{w^T  D^2 \eta(v)
    w}\right|.
\end{equation}
\begin{Rem}
Notice that, in view of~\eqref{eq:entro_commute}, the matrix $D
f_\a(u)$ is self-adjoint for the scalar product 
$
\langle w, v \rangle_{u} = w^T D^2 \eta(u) v.$
Therefore, the Rayleigh quotient 
\be
\label{eq:Rayleigh2}
\sup_{w \in \R^m\setminus\{0\}} 
\left| \frac{w^TD^2 \eta(u)D f_\alpha(u) w}{w^T  D^2 \eta(u) w}\right| 
= \sup_{w \in \R^m\setminus\{0\}} \frac{\langle w, D f_\alpha(u) w
  \rangle_{u}}{\langle w, w \rangle_{u}}
\end{equation}
provides exactly the largest eigenvalue in absolute value of $D f_\alpha(u)$. 
The situation in~\eqref{eq:Lf_Rayleigh} is more intricate than
in~\eqref{eq:Rayleigh2} since $u$ might be different of $v$, but the 
quantity $L_f$ is bounded in 
view of the boundedness of $\O$ and of the regularity of $f_\a$ and $\eta$.
\end{Rem}

Despite it is well-known that even for smooth initial data $u_0$, the solutions 
of~\eqref{eq:P}--\eqref{eq:u0} may develop discontinuities after a finite time, 
our study is restricted to the approximation of smooth solutions
$u \in W^{1,\infty}(\R^d\times\R_+;\O)$
to~\eqref{eq:P}--\eqref{eq:u0}. Such solutions are called \emph{strong
  solutions}, and they satisfy the conservation of the entropy
\begin{equation}\label{eq:entro_strong}
\p_t \eta(u) + \sum_{\a=1}^d \p_\a \xi_\a(u) = 0 \quad \text{ in }
\R^d \times \R_+.
\end{equation}
We refer for instance to~\cite{Kato75, LiTaTsien_book,dafermosBook} 
for specific results on strong solutions of systems of conservation laws.

Assuming that $u_0 \in L^\infty(\R^d;\O)$, a function $u \in
L^\infty(\R^d\times\R_+;\O)$ is said to be a \emph{weak solution} 
to~\eqref{eq:P}--\eqref{eq:u0}
if, for all $\phi \in C^1_c(\R^d \times \R_+;\R^n)$, one has 
\begin{equation}\label{eq:weak}
\iint_{\R^d\times\R_+} u \p_t \phi \; dxdt + \int_{\R^d} u_0 \phi(\cdot,0) \; dx 
+ \iint_{\R^d\times\R_+} \sum_{\a = 1}^d f_\a(u) \p_\a \phi\; dxdt = 0.
\end{equation}
Moreover, $u$ is said to be an \emph{entropy weak solution}
to~\eqref{eq:P}--\eqref{eq:u0}
if $u$ is a weak solution, i.e., $u$ satisfies~\eqref{eq:weak}, and if, for all 
$\psi \in C^1_c(\R^d\times\R_+);\R_+)$, it satisfies 
\begin{equation}\label{eq:entro_weak}
\iint_{\R^d\times\R_+} \eta(u) \p_t \psi dxdt + \int_{\R^d} \eta(u_0) \psi(\cdot,0) dx 
+ \iint_{\R^d\times\R_+} \sum_{\a = 1}^d \xi_\a(u) \p_\a \psi dxdt \ge 0.
\end{equation}

\subsubsection{Relative entropy}

In \cite{Kru70}, Kruzhkov is able to compare two entropy weak
solutions using the doubling variable technique. In  \cite{Kuz76},
such method has been extended in order to compare an entropy weak
solution with an approximate solution.
In the case of systems of conservation laws, these techniques no
longer work.  Basically, the family of entropy--entropy flux pairs
$(\eta,\xi)$ is not sufficiently rich to control the difference
between two solutions.  Nevertheless, let us assume that one of these
solutions is a strong solution, $u$ in the sequel, and introduce:

\begin{Def}[Relative entropy]\label{def:rel_entro}
Let $u,v \in \O$. The relative entropy of $v$ w.r.t. $u$ is defined by 
$$
H(v,u) = \eta(v) - \eta(u) -D\eta(u)(v-u), 
$$
and the corresponding relative entropy fluxes $Q:\Omega\times\Omega
\to \R^d$ are
$$
Q_\a(v,u) = \xi_\a(v) - \xi_\a(u) -  D\eta(u) (f_\a(v)- f_\a(u)), \quad 
\forall \a \in \{1,\dots, d\}. 
$$
\end{Def}

The notion of relative entropy for systems of conservation laws goes
back to the early works of DiPerna and Dafermos (see \cite{DiP79},
\cite{Daf79} and the condensed presentation in \cite{dafermosBook}).
It has also been extensively used for the study of hydrodynamic limits
of kinetic equations (see the first works \cite{Yau91} and
\cite{BGL93}, but also \cite{StR09} for more recent results). For
systems of conservation laws, one can check that, given a strong
solution $u$ and an entropy weak solution $v$ with respective initial
data $u_0$ and $v_0$, one has
\begin{equation}
  \label{eq:EDPsurH}
  \p_t H(v,u) + \sum_{\a =1}^d \p_\a Q_\a(v,u) \leq - \sum_{\a =1}^d
  (\p_\a u)^T Z_\a(v,u)
\end{equation}
in the weak sense, where
\begin{equation}
  \label{eq:Z}
  Z_\a(v,u)= D^2 \eta(u) \bigl( f_{\alpha}(v)-f_{\alpha}(u)-
D f_{\alpha}(u) (v-u) \bigr) .
\end{equation}
On the other hand, it follows from the definition of $H$ that 
\begin{equation}\label{eq:H_int}
H(v,u) = \int_0^1 \int_0^\theta (v-u)^T D^2 \eta(u + \gamma
(v-u)) (v-u) \; d \gamma d \theta, 
\end{equation}
which, together with~\eqref{eq:beta}, leads to 
\begin{equation}\label{eq:Mbeta}
\frac{\beta_0}2 |v-u|^2 \le H(v,u) \le \frac{\beta_1}{2}|v-u|^2, \quad
\forall u,v\in \O. 
\end{equation}
If $u$ is assumed to be a strong solution, its first derivative is
bounded and by a classical localization procedure
\emph{{\`a} la} Kruzhkov and a Gronwall lemma, one obtains a
$L^2_\text{loc}$ stability estimate for any $r>0$
\begin{equation}
  \label{eq:stabL2}
  \int_{|x|<r} |v(x,T)-u(x,T)|^2 dx \leq C(T,u) \int_{|x|<r+ L_fT}
  |v_0(x)-u_0(x)|^2 dx ,
\end{equation}
where the dependence of $C$ on $u$ reflect the needs of smoothness on
$u$ ($C$ blows up when $u$ becomes discontinuous). This inequality,
rigorously proved in \cite{dafermosBook}, provides a weak--strong
uniqueness result.
Similar (but more sophisticated) ideas have been
applied to other fluid systems, see for instance~\cite{lions_book} and \cite{FN12}
for more recent developments. 

\begin{Rem}
  In \cite{Tza05}, Tzavaras studies the comparison of solutions of a
  hyperbolic system with relaxation with solutions of the associated
  equilibrium system of conservation laws. He also makes use of the
  relative entropy for strong solutions.
 Very similar questions have been addressed in
  \cite{BV05,BTV09} for the convergence of kinetic equations towards
  the system of gas dynamics. Here again, only strong solutions of the
  Euler equations are considered. To finish the bibliographical
  review, let us mention the work by Leger and Vasseur \cite{LV11}
  where the reference solution may include some particular
  discontinuities.
\end{Rem}

\begin{Rem}
  \label{rem:Friedrichs}
  For general conservation laws, the relative entropy is not
  symmetric, i.e, $H(u,v) \neq H(v,u)$ and $Q(u,v) \neq Q(v,u)$. In
  the very particular case of \emph{Friedrichs systems}, i.e. when
  there exist symmetric matrices $A_\a \in \R^{m\times m}$ ($\a \in
  \{1,\dots, d\}$) such that $f_\a(u) = A_\a u$, then $u\mapsto |u|^2$
  is an entropy and the corresponding entropy flux $\xi$
  is $\xi_\a(u) = u^T A_\a u$, ($\a \in \{1,\dots, d\}$). It is
  then easy to check that
  $$
  H(v,u) =H(v,u) = |u-v|^2, \qquad   Q_\alpha(v,u) = Q_\alpha(u,v)=(v-u)^T A_\a (v-u),
  $$
  and $Z_\alpha(v,u) = 0$ for all $(u,v) \in \R^m$. As a consequence,
  inequality~\eqref{eq:EDPsurH} becomes
  $$
  \p_t H(v,u) + \sum_{\a =1}^d \p_\a Q_\a(v,u) \leq 0 ,
  $$
  even if $u$ is only a weak solution. This allows to make use of the
  doubling variable technique~\cite{Kru70} to compare $u$ to $v$,
  recovering the classical uniqueness result for Friedrichs systems
  \cite{Fri54}.
\end{Rem}

Our aim is to replace the entropy weak solution $v$
in~\eqref{eq:EDPsurH} by an approximate solution provided by finite
volume schemes on unstructured meshes. Following the formalism
introduced in~\cite{EGGH98}, this makes appear in~\eqref{eq:EDPsurH}
bounded Radon measures which can be controlled, leading to error
estimates in $h^{1/4}$ between a strong solution and its finite volume
approximation, $h$ being the characteristic size of the cells of the
mesh.
The purpose of the following section is to define the finite
volume scheme and to recall some classical properties required on the 
numerical fluxes.

\subsection{Definition of the time-explicit finite volume scheme}

\subsubsection{Space and time discretizations}

Let $\T$ be a mesh of $\R^d$, defined as a family fo disjoint
polygonal (or polyhedral) connected subsets of $\R^d$, such that
$\R^d$ is the union of the closure of the elements of $\T$. We denote
$h=\sup\{\text{diam}(K),\; K\in \T\}<\infty$, and assume without loss
of generality that $0<h\leq1$.  For all $K \in \T$, we denote by $|K|$
its $d$--dimensional Lebesgue measure, and by $\NK$ the set of its
neighboring cells.  For $L\in \NK$, the common interface {(called
  \emph{edge})} between $K$ and $L$ is denoted by $\sigma_{KL}$ and
$|\sigma_{KL}|$ is its $(d-1)$--Lebesgue measure.  We denote by $\Ee$
the set of all the edges and assume that there exists $a>0$ such that
\begin{equation}
\label{eq:reg_mesh}
|K| \ge a h^d \quad \text{ and }\quad |\p K| := \sumNK |\sKL| \le
\frac{h^{d-1}}a, \quad \forall K \in \T. 
\end{equation}
The unit normal vector to $\sigma_{KL}$ from $K$ to $L$ is denoted
$n_{KL}$.  Note that the elements we consider are not
  necessarily simplices.  Let $\dt >0$ be the time step and we set
$t^n = n\dt,~\forall n\in\N$. Let $T>0$ be a given time, we introduce
$N_T=\max\{n\in \N, n\leq T/\dt+1 \}$. Since we consider time-explicit
methods, the time step $\Delta t$ will be subject to a CFL condition
which will be given later.

\begin{Rem}\label{Rem:unif-disc}
  In order to avoid some additional heavy notations, we have chosen to
  deal with a uniform time discretization and a space discretization
  that does not depend on time. Nevertheless, it is possible,
  following the path described in~\cite{ohlberger00}, to adapt our
  study to the case of time-dependent space discretizations and to
  non-uniform time discretizations. This would be mandatory for
  considering a dynamic mesh adaptation procedure based on the \emph{a
    posteriori} numerical error estimators that can be derived from
  our study.
\end{Rem}

Since we will consider weak formulations and compactly supported test
functions in the next sections, we introduce local sets of cells and
interfaces: let $r>0$, we introduce the sets
\begin{equation}
  \label{eq:Tr_Er}
  \begin{aligned}
    \Tr &= \{ K\in\T \ | \ K\subset B(0,r)\}, \\
    \mathcal E_r &= \{ \sKL \in \Ee \ | \ (K,L)\in (\Tr)^2, \; L\in
    \mathcal N(K)\}, \\
    \dTr &= \{ \sKL \in \Ee \ | \ K\in \mathcal T_r, \; L\in
    \mathcal N(K), \; L \not\in \Tr \}.
  \end{aligned}
\end{equation}
In particular, $\{\sKL \in \Ee \ | \ K\in\Tr, L\in\mathcal N(K) \} =
\mathcal E_r\cup \dTr$ and $\mathcal E_r \cap \dTr = \emptyset$.

\subsubsection{Numerical flux and finite volume schemes}\label{sssec:fluxes}

\smallskip

For all $(K,L)\in \T^2$, $L\in\NK$, we consider numerical fluxes
$G_{KL}$, which are Lipschitz continuous functions from $\O^2$ to
$\R^m$.  We assume that these numerical fluxes are
\emph{conservative}, i.e.,
\begin{equation}
  \label{eq:conservation}
  G_{KL}(u,v) = - G_{LK}(v,u),\quad \forall (u,v) \in
  \O^2,
\end{equation}
We also assume that the numerical fluxes fulfill the following
\emph{consistency} condition:
\begin{equation}
  \label{eq:consistance}
  G_{KL}(u,u) = f(u)\cdot n_{KL}, \quad \forall u \in
  \O, 
\end{equation}
which implies 
\begin{equation}
  \label{eq:divnulle}
  \sumNK |\sKL| G_{KL}(u,u) = 0, \quad  \forall u \in \R^m,~\forall K \in \T.
\end{equation}
Following~\cite{bouchut_book}, we assume that the numerical flux
ensures the preservation of the convex set of admissible states
$\Omega$ at each interface. 
More precisely, we assume that there exists $\lambda^\star >0$ such
that, for all $\lambda > \lambda^\star$, for all $K \in \T$, and for
all $L \in \mathcal{N}(K)$,
\begin{equation}
  \label{eq:Omega_stab}
  u - \frac{1}{\lambda}
  \left( G_{KL}(u,v) - f(u)\cdot n_{KL} \right) \in \Omega, \quad
  \forall (u,v) \in \Omega^2.
\end{equation}
In order to ensure the nonlinear stability of the scheme, we also
require the existence of a \emph{numerical entropy flux}. More
precisely, we assume that for all $(K,L)\in \Ee$, there exist
Lipschitz continuous functions $\xi_{KL}:\O\times \O \to \R$ which are
conservative, i.e.,
\begin{equation}
  \label{eq:conservation_fluxentrop}
  \xi_{KL}(u,v) = - \xi_{LK}(v,u), \quad \forall (u,v) \in \O^2,
\end{equation}
and satisfy the interfacial entropy inequalities: for all
$\lambda \ge \lambda^\star>0$, for all $(u,v) \in \O^2$,
\begin{equation}
  \label{eq:bouchut}
  \xi_{KL}(u,v) - \xib(u)\cdot n_{KL}
  \leq 
  -\lambda\bigl(\etab\bigl(u- \frac{1}\lambda
  \big(G_{KL}(u,v)  - f(u)\cdot
  n_{KL}\big)\bigr) -\etab(u)\bigr).
\end{equation}

In what follows, and before strengthening it in~\eqref{eq:CFL_WBV}, we assume 
that the following CFL condition is fulfilled:
\begin{equation}\label{eq:CFL0}
\frac{\dt }{|K|}\lambda^\star \sumNK |\sKL|  \le 1, \quad \forall K \in \T.
\end{equation}
Note that the regularity of the mesh~\eqref{eq:reg_mesh} implies 
that~\eqref{eq:CFL0} holds if 
\begin{equation}\label{eq:CFL}
\dt \le \frac{a^2}{\lambda^\star} h.
\end{equation}

We have now introduced all the necessary material to define the
time-explicit numerical scheme we will consider.
\begin{Def}[Finite volume scheme]
  The finite volume scheme is defined by the discrete unknowns $u_K^n$,
  $K \in\T$ and $n \in \{0,\dots, N_T\}$, which satisfy
  \begin{equation}
    \label{eq:numsch}
    \dfrac{u_K^{n+1}- u_K^n}{\dt}|K| + \sumNK |\sKL| 
    G_{KL}(u_K^n,u_L^n)
    =0
  \end{equation}
  together with the initial condition
  \begin{equation}
    \label{eq:CI}
    u_K^0 = \dfrac{1}{|K|}\int_K u_0(x) dx,\quad  \forall K \in \T,
  \end{equation}
  under assumptions~\eqref{eq:conservation}--\eqref{eq:bouchut} on the
  numerical flux $G_{KL}$ and under the CFL condition~\eqref{eq:CFL}.
  The approximate solution $u^h: \R^d \times \R_+ \to \R^m$ provided
  by the finite volume scheme~\eqref{eq:numsch}--\eqref{eq:CI} is defined by
\begin{equation}
  \label{eq:uh}
  u^h(x,t) = u_K^n,\quad \text{ for } x\in K, ~ \nico{t^n \leq t <
    t^{n+1}}, ~ K \in \T,~ n \in \{0,\dots, N_T\}.
\end{equation}
\end{Def}

\begin{Rem}
  Let us provide some examples of numerical fluxes which satisfy
  assumptions~\eqref{eq:Omega_stab} and~\eqref{eq:bouchut}. The most
  classical example is the Godunov flux \cite{godunov}, which writes
  \begin{equation*}
    G_{KL}(u,v) = f(\mathcal{U}_{KL}(0;u,v))\cdot n_{KL}
  \end{equation*}
  where $\mathcal{U}_{KL}(x/t;u,v)$ stands for the solution of the
  Riemann problem for the system of conservation laws~\eqref{eq:P} in
  the one-dimensional direction $n_{KL}$, with initial data $u$ and
  $v$. If $\lambda^\star$ is greater than all the wave speeds in the
  Riemann problems, then one can prove~\eqref{eq:Omega_stab}
  and~\eqref{eq:bouchut}, with the numerical entropy flux
  \begin{equation*}
    \xi_{KL}(u,v) = \xi(\mathcal{U}_{KL}(0;u,v))\cdot n_{KL} .
  \end{equation*}
  Another classical example is the Rusanov scheme~\cite{rusanov},
  which is the finite volume extension of the Lax--Friedrichs
  scheme. It reads
  \begin{equation*}
     G_{KL}(u,v) = \frac12 (f(u)+f(v))\cdot n_{KL} -
     \frac{c}{2} (v-u)
  \end{equation*}
  where $c>0$ is a parameter (which can be defined by interface). The
  associated numerical entropy flux is
  \begin{equation*}
    \xi_{KL}(u,v) = \frac12 \big(X_{KL}(u,v) -X_{LK}(v,u)\big)
  \end{equation*}
  where $X_{KL}(u,v)=\xi(u)\cdot n_{KL} +
  D\eta(u)(G_{KL}(u,v)-f(u)\cdot n_{KL})$ (this function will also be
  introduced hereafter for the computation of weak-$BV$ estimates).
  Once again, if $c$ is greater than all the wave speeds, one can
  prove that this numerical entropy flux satisfies~\eqref{eq:bouchut}.
  Proving Assumption~\eqref{eq:Omega_stab} is more difficult 
  and overall model dependent. 
 For the shallow-water equations, the positivity of the
  height of water is directly obtained (see for instance~\cite{bouchut_book}). 
  The case of Euler equations is more intricate, in particular for proving the positivity of the
  specific energy. This can be done using the structure of the
  system, see for instance \cite{bouchut_book}.  For details on
  the proofs, more explicit CFL conditions, or for other admissible
  numerical fluxes, the reader can refer for instance to \cite{HLL},
  \cite{coqperth}, \cite{TadmorBook}, \cite{bouchut_book},
  \cite{cgs}.
\end{Rem}

\subsection{Error estimate and organization of the paper}
Our aim is to provide an error estimate of the the form
$$
\| u - u^h \|_{L^2(\Gamma)} \le C h^{1/4},
$$
for all compact subsets $\Gamma$ of $\R^d \times \R_+$, where $u$
stands for the unique strong solution to~\eqref{eq:P}, \eqref{eq:u0}
and $u_h$ for the numerical solution~\eqref{eq:numsch}--\eqref{eq:uh}.
The rigorous statement is given in Theorem~\ref{thm:err}.  This
estimate extends to the system case the contributions of~\cite{CClF94,
  vila94, EGGH98, chainais99} on the scalar case. In \cite{JR06},
which also deals with strong solutions of nonlinear systems, the
assumptions are less classical than ours, in particular we do not need
any 'inverse' CFL condition of the form $C\leq\Delta t/h$ (see also
\cite{EGGH98} for a similar comment in the scalar case).

The proof of this estimate relies on a so-called \emph{weak--$BV$ estimate}, that is
\begin{equation*}
  \sum_{n=0}^{N_T}\Delta t \sum_{(K,L)\in\mathcal E_r} |\sigma_{KL}|
  |G_{KL}(u_K^n, u_L^n) - f(u_K^n)\cdot n_{KL}|\leq \dfrac{C}{\sqrt{h}},
\end{equation*}
where $\mathcal E_r$ is defined in~\eqref{eq:Tr_Er}. The rigorous
statement of this estimate and its proof are gathered in~\S\ref{ssec:WBV}. Up to the
authors' knowledge, this estimate is new for time-explicit finite
volume schemes: in \cite{JR06}, only time-implicit methods
are considered (see also~\cite{HM14}).

Let us now present the outline of the paper.  In
Section~\ref{sec:scheme} we first briefly recall some classical
properties of the finite volume scheme.  Then we address the proof of
the weak--$BV$ property by introducing a new flux which depicts the
entropy dissipation through the edges.  Straightforward consequences
are then derived.

The next two sections address the proof of the error estimate.  In
order to compare the discrete solution $u^h$ with the strong solution
$u$, we write continuous weak and entropy formulations for $u^h$ in
Section~\ref{sec:mesures}, so that we can adapt the uniqueness proof
proposed in~\cite{dafermosBook}.  Nevertheless, the discrete solution
$u^h$ is obviously not a weak entropy solution. Therefore, some error
terms coming from the discretization have to be taken into account in
the formulation, which take the form of positive locally bounded Radon
measures, following~\cite{EGGH98}. A large part of
Section~\ref{sec:mesures} consists in making these measures explicit
and in bounding them with quantities which tend to $0$ with the
discretization size. In Section~\ref{sec:error}, we make use of the
weak and entropy weak formulations for the discrete solution (and of
their corresponding error measures) to derive the error estimate. The
distance between the strong solution $u$ and the discrete solution
$u^h$ is quantified thanks to the relative entropy $H(u^h,u)$
introduced in Definition~\ref{def:rel_entro}.

\section{Nonlinear stability}
\label{sec:scheme}

\subsection{Preservation of admissible states and discrete entropy
  inequality}
We first give two classical properties of the numerical
scheme~\eqref{eq:numsch} which are direct consequences of the
assumptions we made in \S\ref{sssec:fluxes}.  We refer
to~\cite{bouchut_book} for the proofs.

\begin{Lem}\label{lem:Omega_stab}
  Assume that the initial condition satisfies~\eqref{eq:u0} and that
  assumption \eqref{eq:Omega_stab} and the CFL
  condition~\eqref{eq:CFL0} hold, then, for all $K \in \T$, for all $n
  \in \{0,\dots, N_T\}$, $u_K^n$ belong to $\O$.
\end{Lem}
%

%

Following once again the procedure detailed in~\cite{bouchut_book}, we can derive entropy properties on the
numerical scheme from~\eqref{eq:bouchut}. 

\begin{Prop}
  \label{prop:fluxentropnum}
    The numerical entropy flux
  $\xi_{KL}$ is consistent with $\xib$, i.e.
  \begin{equation}
    \label{eq:consistance_fluxentrop}
    \xi_{KL}(u,u) =  \xib(u)\cdot n_{KL}, \quad 
    ~\forall u \in \O.
    \end{equation}
    Moreover, under the CFL condition~\eqref{eq:CFL}, the discrete solution $u^h$ 
    satisfies the discrete entropy inequalities: $\forall K \in \Tt, \,\forall n \ge 0$,
    \begin{equation}    
    \dfrac{|K|}{\dt}(\etab(u_K^{n+1})-\etab(u_K^n)) + \sumNK {|\sigma_{KL}|}
    \xi_{KL}(u_K^n,u_L^n) \leq 0.\label{eq:ineq_entrop_sch}
  \end{equation}
\end{Prop}

%

Note that the consistency~\eqref{eq:consistance_fluxentrop} 
of the entropy fluxes $\xi_{KL}$ ensures that
\begin{equation}
  \label{eq:divnulle_fluxentropnum}
  \sumNK |\sKL|\xi_{KL}(u,u) = 0, \qquad \forall u \in \O.
\end{equation}

\subsection{Weak--$BV$ inequality for systems of conservation laws}\label{ssec:WBV}

For all $(K,L)\in \mathcal T^2$, $L\in \mathcal N(K)$, we introduce
the flux
\begin{equation}
  \label{eq:X_KL}
  X_{KL}(u,v) := \xi(u)\cdot n_{KL} +
  D\eta(u)(G_{KL}(u,v)-f(u)\cdot n_{KL}), \quad \forall
  (u,v)\in \Omega^2.
\end{equation}
Let us remark that it is neither symmetric nor conservative. Such a
quantity may provide the connection between fully discrete and
semi-discrete entropy satisfying schemes, but also between
entropy-conservative and entropy-stable schemes.
It is in particular shown in~\cite{bouchut_book} (see also \cite{Tadmor87, TadmorBook}) 
that the fluxes $X_{KL}$ for $(K,L) \in \Ee$ verify
  \begin{equation}
    \label{eq:Bouchut_etendu}
    -X_{LK} (v,u) \leq \xi_{KL}(u,v) \leq X_{KL}(u,v), \quad \forall
    (u,v)\in \Omega^2.
  \end{equation}
%
Actually, inequalities~\eqref{eq:Bouchut_etendu} can be specified by quantifying the entropy dissipation 
across  the edges. 

\begin{Prop}
  \label{prop:CFL_WBV}
  For all $\sigma_{KL} \in
  \mathcal E$ and all $(u,v)\in \Omega^2$, one has 
  \begin{equation}
    \label{eq:X-xi}
    X_{KL}(u,v) -\xi_{KL}(u,v) \geq
    \frac{\beta_0}{2\lambda^\star}|G_{KL}(u,v) - f(u)\cdot n_{KL}|^2,
  \end{equation}
  where $\beta_0$ is defined in~\eqref{eq:beta} and $\lambda^\star$ has to be 
  such that~\eqref{eq:Omega_stab} and~\eqref{eq:bouchut} hold.
\end{Prop}

\begin{proof}
  We rewrite the left-hand side of Ineq.~\eqref{eq:bouchut} for $\lambda=\lambda^\star$ 
  using the definition~\eqref{eq:X_KL} of the flux $X_{KL}$ in order to obtain
  \begin{multline}
    \label{eq:Xi2}
      X_{KL}(u,v)-\xi_{KL}(u,v)- D\eta(u)(G_{KL}(u,v) - f(u)\cdot
      n_{KL}) \\
      \geq \lambda^\star\left[ \eta(u-\dfrac 1 {\lambda^\star} (G_{KL}(u,v)
        - f(u)\cdot n_{KL})) -\eta(u)\right].
  \end{multline}
  The uniform convexity~\eqref{eq:beta} of $\eta$ ensures that
  \begin{multline}
      \label{eq:eta2}
      \lambda^\star\left[ \eta(u -\dfrac{1}{\lambda^\star} (G_{KL}(u,v) - f(u)\cdot
        n_{KL})) -\eta(u)\right] \\
      \geq -D\eta(u)(G_{KL}(u,v) - f(u)\cdot
      n_{KL}) + \dfrac 1 2 \dfrac{\beta_0}{\lambda^\star}|G_{KL} (u,v)-
      f(u)\cdot n_{KL}|^2.
  \end{multline}
  Combining~\eqref{eq:Xi2} and~\eqref{eq:eta2} leads to~\eqref{eq:X-xi}.
%
\end{proof}

Thanks to the specified version \eqref{eq:X-xi} of the classical
inequalities~\eqref{eq:Bouchut_etendu}, we are now in position for
proving a new stability estimate for time-explicit finite volume
scheme, namely the weak-$BV$ inequality. This inequality is obtained
by quantifying the numerical diffusion of the numerical scheme. As in
the scalar case (see~\cite{CGH93,CClF94,vila94,EGGH98,chainais99}),
such an equality requires a strengthened CFL condition. In our system
case, we require the existence of some $\zeta \in (0,1)$ such that
\begin{equation}
  \label{eq:CFL_WBV}
  \Delta t \leq \dfrac{\beta_0}{\beta_1}\dfrac{a^2}{\lambda^\star}(1-\zeta)h
\end{equation}
holds, where $\beta_0$ and $\beta_1$ are defined by~\eqref{eq:beta},
$a$ and $h$ are the mesh parameters~\eqref{eq:reg_mesh}, and where
$\lambda^\star$ appears in the condition~\eqref{eq:Omega_stab}
and~\eqref{eq:bouchut}. Note that the strengthened CFL
condition~\eqref{eq:CFL_WBV} implies the classical CFL
condition~\eqref{eq:CFL}.
%
We are now able to obtain the following local estimate, using the notations~\eqref{eq:Tr_Er}.
\begin{Prop}
  \label{Prop:WBV2_ln}
  Assume that the strengthened CFL condition~\eqref{eq:CFL_WBV} holds,
  then there exists $C$ depending only on $T,r,a,\eta,\xi,\Omega$ and
  $\zeta$ (but neither on $h$ nor on $\dt$) such that
  \begin{equation}
    \label{eq:WBV2_ln}
    \sum_{n=0}^{N_T}\Delta t \sum_{(K,L)\in \mathcal E_r}
    |\sigma_{KL}| \ |G_{KL}(u_K^n,u_L^n) 
    -f(u_K^n)\cdot n_{KL}|^2\leq C.
  \end{equation}
\end{Prop}

\begin{proof}
  Multiplying the numerical scheme~\eqref{eq:numsch} by $\dt D
  \eta(u_K^n)$ and summing over $n \in \{0,\dots N_T\}$ and $K \in
  \Tr$ provides
  \begin{equation}
    \label{eq:AB_ln}
    A+ B = 0,
  \end{equation}
  where 
  \begin{align*}
    A = & \sum_{n=0}^{N_T}\sum_{K \in \Tr}  D \eta(u_K^n) (u_K^{n+1}
    - u_K^n) |K|, \\
    B = & \sumk \dt \sumK D \eta(u_K^n) \sumNK |\sKL|
    G_{KL}(u_K^n, u_L^n).
  \end{align*}
  The concavity of $u\mapsto \eta(u) - \frac{\beta_1}2 |u - u_K^n|^2$
  together with the definition~\eqref{eq:numsch} of the numerical
  scheme and property~\eqref{eq:divnulle} provide that
%
  \begin{align*} 
    A \ge & \sumK \eta(u_K^{N_T+1}) |K| -  \sumK \eta(u_K^{0}) |K| \\
    \nn	& - \frac{\beta_1}2 \sumk \dt^2 \sumK \frac1{|K|} \bigg| \sumNK
      |\sKL| \left( G_{KL}(u_K^n , u_L^n) - f(u_K^n)\cdot n_{KL} \right)
    \bigg|^2.
\end{align*}
Using the Jensen inequality, we get 
\begin{equation*}
\sumK \eta(u_K^{0}) |K| \le \int_{|x| \le B(0,R+h)} \eta(u_0(x)) dx
=: C_1.
\end{equation*}
The positivity of the entropy $\eta$ yields 
$
\displaystyle  \sumK \eta(u_K^{N_T+1}) |K| \ge0.
$
Moreover, Cauchy--Schwarz inequality ensures that for all $K \in \Tr$
and all $n \in \{0,\dots , N_T\}$, one has 
\begin{multline*}
  \bigg| \sumNK |\sKL| \left( G_{KL}(u_K^n , u_L^n) - f(u_K^n)\cdot n_{KL}
    \right) \bigg|^2 \\
  \le \bigg( \sumNK |\sKL| \bigg)\bigg( \sumNK |\sKL| \left|
      G_{KL}(u_K^n , u_L^n) - f(u_K^n)\cdot n_{KL} \right|^2 \bigg).
\end{multline*}
Then it follows from the regularity assumption~\eqref{eq:reg_mesh} on the 
mesh that 
\begin{equation}\label{eq:A_ln}
  A \ge - C_1
  - \frac{\beta_1\dt}{2a^2 h} \sumk \dt  \sumK \sumNK |\sKL| \left|
    G_{KL}(u_K^n , u_L^n) - f(u_K^n)\cdot n_{KL} \right|^2.
\end{equation}
Concerning the term $B$, we use the definition~\eqref{eq:X_KL} of the
entropy flux $X_{KL}$ to get
\begin{equation*}
  B = \sumk \dt \sumK \sumNK |\sKL|
    (X_{KL}(u_K^n, u_L^n) - \xi(u_K^n)\cdot n_{KL} +
    D\eta(u_K^n)f(u_K^n)\cdot n_{KL}).
\end{equation*}
Using the property $\sumNK |\sKL|n_{KL}=0$ for all $K\in \mathcal T$, 
we can reorganize the term $B$ into 
\begin{equation}\label{eq:B12}
B = B_1 + B_2, 
\end{equation}
where 
\begin{align*}
B_1 = &\sumk \dt \sumK \sumNK |\sKL| (X_{KL}(u_K^n,
    u_L^n)-\xi_{KL}(u_K^n,u_L^n)), \\
B_2 = &\sumk \dt  \sum_{(K,L)\in\dTr}  |\sKL|\xi_{KL}(u_K^n,u_L^n).
\end{align*}
Since $\xi_{KL}$
is a continuous function of bounded quantities, $B_2$ 
can be bounded using the regularity of the mesh~\eqref{eq:reg_mesh}.
More precisely, one gets
\begin{equation}\label{eq:C2}
|B_2|
    \leq
    \max_{(K,L)\in\dTr} \|\xi_{KL}\|_{L^\infty(\Omega^2)} \sumk \dt
    \sum_{(K,L)\in\dTr} |\sKL| \leq  C_2
\end{equation}
for some $C_2>0$ depending only on $T$, $r$, $a$, $\xi$ and $\Omega$.
On the other hand, it follows from Proposition~\ref{prop:CFL_WBV} that 
\begin{equation}
  \label{eq:B_ln2}
    B_1\geq \frac{\beta_0}{2\lambda^\star}\sumk \dt \sumK \sumNK |\sKL| \left.
      |G_{KL}(u_k^n,u_L^n) - f(u_K^n)\cdot n_{KL}|^2\right..
\end{equation}
Combining~\eqref{eq:A_ln}--\eqref{eq:B_ln2} into~\eqref{eq:AB_ln}
leads to
$$
\left( \frac{\beta_0}{2\lambda^\star} - \dfrac{\beta_1\Delta t}{2 a^2 h}\right) \sumk \dt
    \sumK \sumNK |\sKL|
      |G_{KL}(u_k^n,u_L^n) - f(u_K^n)\cdot n_{KL}|^2 \le C_1 + C_2.
$$
The CFL condition~\eqref{eq:CFL_WBV} has been strengthened so that 
$$
 \left( \frac{\beta_0}{2\lambda^\star} - \dfrac{\beta_1\Delta t}{2 a^2 h}\right) \ge \dfrac{\zeta \beta_0}{2\lambda^\star}
 $$
 remains uniformly bounded away from $0$. Estimate~\eqref{eq:WBV2_ln}  follows.
\end{proof}

We state now a straightforward consequence of
Proposition~\ref{Prop:WBV2_ln}.  Its proof relies on the
Cauchy--Schwarz inequality and is left to the reader.
\begin{Coro}\label{Coro:wBV_ln}
  Assume that~\eqref{eq:CFL_WBV} holds, then
  there exists $C_{BV}$ depending only on $T$, $r$, $a$, $\xi$,
  $\eta$, $u_0$, $\O$ and $\zeta$ such that
\begin{equation}
\label{eq:weak-BV}
\sumk \dt \sum_{(K,L)\in \mathcal E_r}  |\sigma_{KL}| \left|G_{KL}(u_K^n,u_L^n) -
    f(u_K^n)\cdot n_{KL}\right|
  \leq \frac{C_{BV}}{\sqrt h}.
\end{equation}
\end{Coro}

\subsection{Consequences of the weak--$BV$ estimate}
\label{ssec:consequences}

The weak--$BV$ estimate~\eqref{eq:weak-BV} implies
a similar control on entropy fluxes and the time variations of $u^h$.
\begin{Lem}\label{Lem:BV++}
  Assume that the strengthened CFL condition~\eqref{eq:CFL_WBV} holds, then
  \begin{align}
    \label{eq:BVespace_phi}
    &\sum_{n=0}^{N_T} \dt \sumE |\sKL| \left| \xi_{KL}(u_K^n, u_L^n) -
      \xi(u_K)\cdot n_{KL} \right| \le \ninf{{D}
      \etab}\dfrac{C_{BV}}{\sqrt{h}}, \\
    \label{eq:BVtime_u}
    &
    \sumk \sumK |K| |u_K^{n+1}-u_K^n| \leq \dfrac{C_{BV}}{\sqrt{h}},\\
    \label{eq:BVtime_phi}
    &\sumk \sumK |K| |\eta(u_K^{n+1})-\eta(u_K^n)| \leq \ninf{{D}
      \eta}\dfrac{C_{BV}}{\sqrt{h}}.
  \end{align}
\end{Lem}
\begin{proof}
  Using the Lipschitz continuity of $\etab$ in~\eqref{eq:bouchut}, one
  obtains inequality~\eqref{eq:BVtime_phi}. Thanks to
  definition~\eqref{eq:numsch} of the scheme and thanks to the
  divergence free property~\eqref{eq:divnulle}, one has for all $K \in
  \T$ and all $n \in \N$
$$
{|u_K^{n+1}-u_K^n|} |K| \le   \dt \sumNK |\sKL| \left| G_{KL}(u^n_K,u^n_L) - f(u^n_K) \cdot n_{KL}\right|.
$$
Summing over $K \in \T_r$ and $n \in \{0,\dots, N_T\}$ and
using~\eqref{eq:weak-BV} provides~\eqref{eq:BVtime_u}.
Inequality~\eqref{eq:BVtime_phi} then follows from the Lipschitz
continuity of $\etab$.
\end{proof}

We now state our main result, that consists in an \emph{a priori} error estimate 
between a strong solution $u$ and a discrete solution $u^h$.

\begin{Thm}\label{thm:err}
  Assume that $u_0\in W^{1,\infty}(\R^d)$ and that the solution $u$ of
  the Cauchy problem \eqref{eq:P}--\eqref{eq:u0} belongs to
  $W^{1,\infty}(\R^d\times[0,T])$.  Let $u^h$, with $0<h\leq1$,
  defined by the numerical scheme~\eqref{eq:numsch}--\eqref{eq:uh} and
  assume that the strengthened CFL condition~\eqref{eq:CFL_WBV} holds.
  Then, for all $r>0$ and $T>0$ there exist $C$ depending only on
  $T,r,\O,a,\lambda^\star,u_0,G_{KL}, \eta$ and $f$, such that
  $$
  \int_0^T \int_{B(0, r+L_f(T-t))} |u-u^h|^2 dxdt  \le C\sqrt{h}.
  $$
\end{Thm}

\section{Continuous weak and entropy formulations for the discrete solution}
\label{sec:mesures}

In order to obtain the error estimate of Theorem~\ref{thm:err},
  we aim at using the relative entropy of $u^h$ w.r.t. $u$. Since
  $u^h$ is only an approximate solution, it neither satisfies exactly 
   the weak formulation~\eqref{eq:weak} nor the entropy weak
  formulation~\eqref{eq:entro_weak}. Some numerical error terms appear
  in these formulations, and thus also appear the inequality of the
  relative entropy
\begin{equation}
  \label{eq:Huhuapp}
  \begin{aligned}
    \p_t H(u^h,u) + \sum_{\a =1}^d \p_\a Q_\a(u^h,u) &\leq - \sum_{\a
      =1}^d (\p_\a u)^T Z_\a(u^h,u) \\
    &\quad + \text{numerical error terms} .
  \end{aligned}
\end{equation}
As usual, these terms may be described by Radon measures, see for
instance \cite{BP98,EGGH98,chainais99,ohlberger00,JR05,JR06}. Note
that for nonlinear systems of conservation laws, a function which
satisfies the entropy inequality~\eqref{eq:entro_weak} is not
necessarily a weak solution~\eqref{eq:weak}. This leads us to introduce 
error measures for both the entropy inequality~\eqref{eq:entro_weak} 
and the weak formulation~\eqref{eq:weak} of $u^h$. 
Let us first begin with the entropy formulation and the related measures.

For $X=\R^d$ or $X =\R^d\times \R^+$, we denote by $\Mm(X)$ the set of
locally bounded Radon measures on $X$, i.e., $\Mm(X) = \left(
  C_c(X)\right)'$ where $C_c(X)$ is the set of continuous compactly 
supported functions on $X$.  If $\mu \in \Mm(X)$ we set
$
  \langle \mu, \varphi\rangle = \int_{X} \varphi d\mu
  $
  for all $\varphi \in C_c(X)$.

\begin{Def}
  \label{def:mesures_mu_mu0}
  For $\psi  \in C_c(\R^d)$, $\varphi  \in C_c(\R^d\times \R^+)$, we define $\mu_0 \in \Mm(\R^d)$ 
  and $\mu \in \Mm(\R^d\times \R^+)$ by
\begin{align*}
      \langle \mu_0,\psi \rangle =& \intR
      |\eta(u_0(x))-\eta(u^h(x,0))|\psi(x)dx,\\
      \langle \mu,\varphi \rangle =& \langle \mu_{\T},\varphi\rangle +
      \sum_{n=0}^\infty \dt\sumE |\sKL| \left|\xi_{KL}(u_K^n, u_L^n) - \xi_{KL}(u_K^n,u_K^n)\right|
      \langle \mu_{KL},\varphi \rangle \\
      & + \sum_{n=0}^\infty \dt \sumE |\sKL| \left|\xi_{KL}(u_K^n, u_L^n)
      -\xi_{KL}(u_L^n,u_L^n)\right| \langle \mu_{LK},\varphi \rangle,
  \end{align*}
  where
  \begin{align*}
      \langle \mu_\T, \varphi&\rangle = \sum_{n=0}^\infty\sumK
      |\eta(u_K^{n+1})-\eta(u_K^n)|\intt\int_{K}
      \varphi(x,t)dxdt,\\
      \langle \mu_{KL},\varphi&\rangle =
      \frac{1}{|K|\ |\sigma_{KL}|\ (\dt)^2} \\
      &\times\intt \hspace{-7pt}\int_K  \intt \hspace{-7pt}
      \int_{\sigma_{KL}} \hspace{-3pt}\int_0^1\hspace{-3pt}(h+\dt)
      \varphi(\gamma+\theta(x-\gamma), s+ \theta(t-s))
      d\theta dx dt d\gamma ds,\\
      \langle \mu_{LK},\varphi&\rangle =
      \frac{1}{|K|\ |\sigma_{KL}|\ (\dt)^2} \\
      &\times\intt \hspace{-7pt}\int_L \intt \hspace{-7pt}
      \int_{\sigma_{KL}} \hspace{-3pt}\int_0^1\hspace{-3pt} (h+\dt)
      \varphi(\gamma+\theta(x-\gamma), s+ \theta(t-s))
      d\theta dx dt d\gamma ds.
  \end{align*}
\end{Def}

As it will be highlighted by Proposition~\ref{lem:entropy_estimate} later on,  
the measures $\mu$ and $\mu_0$  describe the  approximation error
in the entropy formulation satisfied by $u^h$. Let us first estimate them on compact 
sets.

\begin{Lem}
  \label{lem:borne_mu_mu0}
  Assume that the strengthened CFL condition~\eqref{eq:CFL_WBV} holds,
  then, for all $r>0$ and $T>0$ there exist $C_{\mu_0}>0$, depending
  only on $u_0$, $\left\|D \etab\right\|_{\infty}$, and $r$, and
  $C_{\mu}>0$, depending only on $T,r,a,\lambda^\star,u_0,G_{KL}$ and
  $\etab$ such that, for all $h <r$,
\begin{equation}\label{eq:borne_mu}
 \mu_0(B(0,r))\leq C_{\mu_0}h\quad\text{and}\quad \mu(B(0,r)\times[0,T]) \leq \dfrac{C_{\mu}}{\sqrt{h}}.
\end{equation}
\end{Lem}

\begin{proof}
The regularity of  $u_0:\R^d \to \R^m$ yields 
 \begin{equation*}
    \mu_0(B(0,r)) \leq  h\left\|D \etab\right\|_{\infty}
    \int_{B(0,r+h)} {| \nabla u_0 |} dx.  
  \end{equation*}
 For $r>0$ and $T>0$ the measure $\mu_\T$ satisfies
  \begin{equation*}
    \mu_\T (B(0,r)\times[0,T]) = \int_0^T\int_{B(0,r)} \sumk\sumK
    |\etab(u_K^{n+1}) - \etab(u_K^n)| \mathbf 1
    _{K\times[t^n,t^{n+1}]}dx dt.
  \end{equation*}
  Then, using the time--$BV$ estimate \eqref{eq:BVtime_phi},
  \begin{equation*}
   \mu_\T (B(0,r)\times[0,T]) \leq \dt \sumk\sumK |K| |\etab(u_K^{n+1})
   -\etab(u_K^n)|\leq \dt \ninf{D\etab} \dfrac{C_{BV}}{\sqrt{h}}.
  \end{equation*}
  Since $\Delta t$ satisfies the CFL condition~\eqref{eq:CFL_WBV}, one
  has
   \begin{equation}\label{eq:borne_muT}
   \mu_\T (B(0,r)\times[0,T]) \leq C_{\mu_\T}  {\sqrt{h}},
   \end{equation}
   where
   $
   C_{\mu_\T} :=  \frac{a^2\ninf{D\eta}}{\lambda^\star}C_{BV}.
   $
 The measures $\mu_{KL}$ and $\mu_{LK}$ satisfy:
  \begin{equation*}
    \mu_{KL}(\R^d\times\R^+) \leq h+ \dt, \qquad
    \mu_{LK}(\R^d\times\R^+) \leq h+ \dt.
  \end{equation*}
  Therefore, 
  \begin{equation*}
    \begin{aligned}
      \mu(B(0,r)\times[0,T]) & \\
      \leq C_{\mu_{\T}}{\sqrt{h}} 
       +
      (h&+\dt) \sumk \dt \sumE |\sKL| \left|\xi_{KL}(u_K^n, u_L^n) - \xi_{KL}(u_K^n,u_K^n)\right| \\
       + (h&+\dt) \sumk  \dt \sumE |\sKL|\left|\xi_{KL}(u_K^n, u_L^n)-\xi_{KL}(u_L^n,u_L^n)\right| .
    \end{aligned}
  \end{equation*}
  Hence, using Lemma~\ref{Lem:BV++}, the CFL 
  condition~\eqref{eq:CFL} and the bound \eqref{eq:borne_muT} provides 
  \begin{equation*}
    \mu(B(0,r)\times[0,T]) \leq C_{\mu} {\sqrt{h}},
  \end{equation*}
  where $C_{\mu} = C_{\mu_{\T}} + 2\left(1+\frac{a^2}{\lambda^\star}\right) \ninf{D \eta} C_{BV}$.
\end{proof}

\begin{Prop}
  \label{lem:entropy_estimate}
  Let
  $\mu$ and $\mu_0$ be the measures introduced 
  in Definition~\ref{def:mesures_mu_mu0},
  then, for all $\varphi \in C_c^1(\R^d\times\R^+;\R^+)$, one has 
  \begin{multline}
    \label{eq:ineq_entropy}
      \intRR \etab(u^h)\p_t\varphi(x,t) + \suma
      +\intR\etab(u_0(x))\varphi (x,0)dx
      \xib_{\alpha} (u^h) \p_{\alpha}\varphi(x,t) dx dt \\
      \geq - \intRR \left(|\nabla \varphi| + |\p_t
      \varphi|\right)d\mu (x,t) -\intR \varphi(x,0) d\mu_0(x).
  \end{multline}
\end{Prop}

\begin{proof}
  Let $\varphi \in C_c^1(\R^d\times\R^+;\R^+)$. Let $T>0$ and
  $r>0$ such that ${\rm supp}\; \varphi \subset B(0,r)\times[0,T)$.
  Let us multiply \eqref{eq:ineq_entrop_sch} by
  $\intt\int_K \varphi(x,t)dx dt$ and sum over the control volumes
  $K\in \T_r$ and $n\leq N_T$. It yields 
  \begin{equation}\label{eq:T1+T2}
  T_1+T_2 \leq 0,
  \end{equation}
  where
  \begin{align}
    \label{eq:T1}
    T_1 &= \sumk\sumK \dfrac{1}{\dt}(\etab(u_K^{n+1})
    - \etab(u_K^n)) \intt \int_K \varphi(x,t)dxdt, \\
    \label{eq:T2}
    T_2 &= \sumk\sumK \dfrac{1}{|K|}
    \intt \int_K \varphi(x,t)dxdt \sumNK |\sKL| \xi_{KL}(u_K^n,u_L^n).
  \end{align}
  The term $T_1$ corresponds to the discrete time derivative of
  $\etab(u^h)$ and $T_2$ to the discrete space derivative of
  $\xib(u^h)$. The proof relies on the comparison firstly between $T_1$
  and $T_{10}$ and secondly between $T_2$ and $T_{20}$, where $T_{10}$
  and $T_{20}$ denote respectively the temporal and spatial term in
  \eqref{eq:ineq_entropy}:
  \begin{equation*}
    \begin{aligned}
      T_{10} &= - \intRR \etab(u^h) \p_t
      \varphi(x,t) dx dt-\intR \etab(u_0(x))\varphi (x,0)dx, \\
      T_{20} &= - \intRR \suma \xib_{\alpha}
      (u^h) \p_{\alpha}\varphi(x,t) dx dt.
    \end{aligned}
  \end{equation*}

  Let us first focus on $T_{10}$. Following its definition~\eqref{eq:uh}, the 
  approximate solution $u^h$ is piecewise constant, then so does $\eta(u^h)$. 
 Therefore, we can rewrite
  \begin{multline*}
      T_{10} 
      = \sumk\sumK
      (\etab(u_K^{n+1})-\etab(u_K^n))\dfrac{1}{\dt}\intt\int_K
      \varphi(x,t^{n+1})dxdt \\
      -\intR(\etab(u_0(x))-\etab(u^h(x,0)))\varphi(x,0)dx.
  \end{multline*}
It is now easy to verify that 
  \begin{multline*}
      |T_1-T_{10}| 
      \leq  \sumk\sumK
      |\etab(u_K^{n+1})-\etab(u_K^n)|\intt\int_K|\p_t \varphi|dxdt \\
       +\intR |\etab(u_0(x))-\etab(u^h(x,0))|\varphi(x,0)dx.
    \end{multline*}
  Then, accounting from Definition \ref{def:mesures_mu_mu0}, the
  inequality reads
  \begin{equation}
    \label{eq:T1-T10}
    |T_1-T_{10}|\leq \intRR|\p_t \varphi|d\mu_\T(x,t) +
    \intR\varphi(x,0)d\mu_0(x).
  \end{equation}
 
  We now consider the terms $T_2$ and $T_{20}$.
 Performing a discrete integration by parts by reorganizing the sum, 
 and using the properties~\eqref{eq:divnulle_fluxentropnum} and~\eqref{eq:conservation_fluxentrop} lead to
  \begin{equation}\label{eq:T2=T21+T22}
    T_2 = T_{2,1} + T_{2,2},
  \end{equation}
  with
  \begin{equation*}
    \begin{aligned}
      T_{2,1} =& \sumk\sumE \frac{|\sKL|}{|K|}\intt \int_K
      \varphi(x,t) (\xi_{KL}(u_K^n, u_L^n) -
      \xi_{KL}(u_K^n,u_K^n)) dx dt,\\
      T_{2,2} =& \sumk\sumE \frac{|\sKL|}{|L|}\intt \int_L
      \varphi(x,t)  (\xi_{LK}(u_L^n, u_K^n) -
      \xi_{LK}(u_L^n,u_L^n)) dx dt.\\
    \end{aligned}
  \end{equation*}
  Gathering terms of $T_{20}$ by edges yields
  \begin{equation*}
    T_{20 }= T_{20,1}+T_{20,2},
  \end{equation*}
  where, thanks to~\eqref{eq:conservation_fluxentrop}, we have set
  \begin{equation*}
    \begin{aligned}
      T_{20,1} & = \sumk\sumE \intt \int_{\sigma_{KL}} 
      \left( \xi_{KL}(u_K^n,u_L^n) - 
        \xib(u_K^n)\cdot n_{KL}\right) \varphi(\gamma,t)d\gamma dt, \\
      T_{20,2} & = \sumk\sumE \intt \int_{\sigma_{KL}}
      \left( \xi_{LK}(u_L^n,u_K^n) -
              \xib(u_L^n)\cdot n_{LK}\right) \varphi(\gamma,t)d\gamma dt.
    \end{aligned}
  \end{equation*}
It is easy to verify  
  \begin{equation*}
    \begin{aligned}
      T_{2,1}-T_{20,1} =& \sumk \dt \sumE |\sKL|\left(\xi_{KL}(u_K^n, u_L^n) -
      \xi_{KL}(u_K^n, u_K^n)\right) \\
      & \times\frac{1}{|K| |\sigma_{KL}|(\dt)^2}  \intt
      \int_K \intt \int_{\sigma_{KL}} (\varphi(x,t) -
      \varphi(\gamma,s) ) d\gamma ds dx dt.
    \end{aligned}
  \end{equation*}
  Then using the definition of
  $\mu_{KL}$ in Definition \ref{def:mesures_mu_mu0}, we obtain the following estimate:
  \begin{multline}\label{eq:T21-T201}
      |T_{2,1} - T_{20,1}| \\
      \leq \sumk \dt \sumE
      |\sKL|\left|\xi_{KL}(u_K^n, u_L^n) - \xi_{KL}(u_K^n,
        u_K^n)\right| \langle \mu_{KL}, |\nabla \varphi| + |\p_t
      \varphi|\rangle.
  \end{multline}
 Similarly, one obtains
  \begin{multline}\label{eq:T22-T202}
      |T_{2,2} - T_{20,2}| \\
      \leq \sumk \dt \sumE |\sKL| \left|\xi_{LK}(u_L^n,
      u_K^n) -\xi_{LK}(u_L^n, u_L^n)\right| \langle \mu_{LK}, |\nabla
      \varphi| + |\p_t \varphi| \rangle,
  \end{multline}
  the measure $\mu_{LK}\in \mathcal M (\R^d \times \R^+)$ being 
  given by Definition \ref{def:mesures_mu_mu0}.
  Bearing in mind the definition of $\mu\in \mathcal M (\R^d \times
  \R^+)$ given in Definition \ref{def:mesures_mu_mu0}, inequalities~\eqref{eq:T1+T2}, 
  \eqref{eq:T1-T10}, \eqref{eq:T2=T21+T22}, \eqref{eq:T21-T201}  
  and~\eqref{eq:T22-T202}, one has
  \begin{equation*}
    -T_{10}-T_{20} \geq -\intRR (|\nabla\varphi|+|\p_t \varphi|)
    d\mu(x,t) -\intR \varphi(x,0)d\mu_0(x),
  \end{equation*}
  which concludes the proof of Proposition~\ref{lem:entropy_estimate}.
\end{proof}

Similar calculations can be used to estimate how close $u^h$ is to a weak solution.
For that purpose we define the following  measures.
\begin{Def}
  \label{def:mesures_mub_mubo}
For $\psi \in C_c(\R^d)$ and $\varphi\in C_c(\R^d\times \R^+)$,
  we set
  \begin{equation*}
    \label{eq:mub_mubo}
    \begin{aligned}
      \langle \mubo, \psi \rangle &= \intR
      {|u_0(x)-u^h(x,0)|} \psi(x)dx,\\
      \langle \mub, \varphi \rangle &= \langle\mub_\T,
      \varphi \rangle + \sum_{n=0}^\infty \dt \sumE  |\sKL|{|G_{KL}(u_K^n,u_L^n) -
      G_{KL}(u_K^n,u_K^n)|}\langle  \mub_{KL}, \varphi
      \rangle\\
      &\quad+ \sum_{n=0}^\infty\dt \sumE  |\sKL|{|G_{KL}(u_K^n,u_L^n) -
      G_{KL}(u_L^n,u_L^n)|} \langle \mub_{LK}, \varphi
      \rangle,
    \end{aligned}
  \end{equation*}
      where
  \begin{equation*}
    \begin{aligned}
      \langle \mub_\T, \varphi &\rangle = \sum_{n=0}^\infty\sumK
      {|u_K^{n+1}-u_K^n|} \intt\int_K \varphi(x,t) dx dt,\\
      \langle \mub_{KL}, \varphi &\rangle =
      \dfrac{1}{|K||\sigma_{KL}|\dt^2} \times\\
      &\intt\!\! \int_K \intt \!\! \int_{\sigma_{KL}} \int_0^1 (h+ \dt)
      \varphi(\gamma + \theta(x-\gamma), s+ \theta(t-s)) d\theta
      dx dt d\gamma ds,\\
      \langle \mub_{LK}, \varphi &\rangle=
      \dfrac{1}{|L||\sigma_{KL}|\dt^2} \times \\
      &\intt \int_L \intt \int_{\sigma_{KL}} \int_0^1 (h+ \dt)
      \varphi(\gamma +\theta(x-\gamma), s+ \theta(t-s)) d\theta
      dx dt d\gamma ds.
    \end{aligned}
  \end{equation*}
\end{Def}
\begin{Rem}\label{rem:mu_mub}
It follows from the definitions of the measures $\mu$ and $\mub$ that they can be 
extended (in a unique way) into continuous linear forms defined on the set 
$$
E:=\left\{ 
\varphi \in L^\infty(\R^d \times \R^+;\R) \; | \; {\rm supp}(\varphi)
\text{ is compact, and }  \nabla \varphi \in {L^1_{\rm loc}(\R^d
  \times \R^+)}^d
\right\}.
$$
Indeed, any $\varphi \in E$ admit a unique trace on $\sKL$, so that the quantities 
$\langle \mu_{KL}, \varphi \rangle$, $\langle \mu_{LK}, \varphi \rangle$, $\langle \mub_{KL}, \varphi \rangle$ 
and $\langle \mub_{LK}, \varphi \rangle$ are well defined.
Moreover, one has 
$$
\left| \langle \mu,\varphi \rangle \right| \le {\| \varphi \|}_{L^\infty} \mu(\{\varphi \neq 0\}), 
\qquad 
\left| \langle \mub,\varphi \rangle \right| \le {\| \varphi \|}_{L^\infty} \mub(\{\varphi \neq 0\}), 
\qquad 
\forall \varphi \in E.
$$
\end{Rem}

We now state a lemma and a proposition whose proofs are left to the reader, since they are similar to the proofs of 
Lemma~\ref{lem:borne_mu_mu0} and Proposition~\ref{lem:entropy_estimate} respectively as one uses the estimates~\eqref{eq:weak-BV} 
and~\eqref{eq:BVtime_u} instead of~\eqref{eq:BVespace_phi} and~\eqref{eq:BVtime_phi}.

\begin{Lem}\label{lem:bornes_mub_mub0}
  Let $u^h$
  defined by \eqref{eq:numsch}--\eqref{eq:uh}. Assume that~\eqref{eq:CFL_WBV} holds,
  then, for all $r>0$ and $T>0$ there exist $C_{\mub_0} >0$, 
  depending only on $u_0$ and $r$, and
  $C_{\mub}>0$, depending only on $T,r,a,\lambda^\star,u_0,G_{KL}$
  such that, for all $h <r$,%
  \begin{equation}
    \label{eq:mub_mubo2}
      \mubo(B(0,r)) \leq  C_{\mub_0} h \quad \text{and}\quad 
      \mub (B(0,r) \times [0,T]) \leq C_{\mub} {\sqrt{h}},
  \end{equation}
where $C_{\mub_0} = C_{\mu_0}/\ninf{{D}\etab}$ and 
$
C_{\mub} = C_{\mu}/\ninf{{D}\etab} 
$
(see the proof of 
Lemma~\ref{lem:borne_mu_mu0}).
\end{Lem}

We are now in position to provide the approximate weak formulation
satisfied by $u^h$.
In the statement below, $\varphi$ is a vector-valued
  function, and we adopted the notation
  $
    |\nabla \varphi| = \max_{\a \in \{1,\dots, d\}} | \p_\a \varphi | .
    $
The proof of Proposition~\ref{lem:laxwendroff} follows the same guidelines 
  as the proof of Proposition~\ref{lem:entropy_estimate} and is left to the reader.

\begin{Prop}
  \label{lem:laxwendroff}
  Let
  $\mu$ and $\mu_0$ be the measures introduced 
  in Definition~\ref{def:mesures_mu_mu0},
  then, for all $\varphi \in C_c^1(\R^d\times\R^+;\R^m)$, one has 
  \begin{multline*}
      \Bigg|\intRR \left[(u^h)^T \p_t\varphi(x,t)+  \suma f_{\alpha}(u^h)^T \p_{\alpha} \varphi(x,t) \right] dx dt  
      + \intR u_0(x)^T \varphi(x,0)dx \Bigg|
      \\
      \leq \intRR (|\nabla \varphi| + |\p_t\varphi|)d\mub(x,t)
      +\intR|\varphi(x,0)|d\mubo(x).
  \end{multline*}
\end{Prop}

\section{Error estimate using the relative entropy}
\label{sec:error}

With the error measures $\mu$, $\mu_0$, $\ov \mu$, and $\ov \mu_0$ at hand, 
  we are now in position to precise 
  inequality~\eqref{eq:Huhuapp} satisfied by the relative entropy
  $H(u^h,u)$ and then to conclude the proof of Theorem~\ref{thm:err}.  

\subsection{Relative entropy for approximate solutions}

\begin{Prop}
  Let $\mu$ and $\mu_0$ be the measures introduced 
  in Definition~\ref{def:mesures_mu_mu0}, and let $\mub$ and $\mub_0$ 
  be the measures introduced  in Definition~\ref{def:mesures_mub_mubo},
  then, for all $\varphi \in C_c^1(\R^d\times\R^+;\R^+)$, one has 
    \begin{multline}      \label{eq:ineq_relative_entrop}\intRR \left(H(u^h,u)\p_t\varphi (x,t) + \suma
      Q_{\alpha}(u^h,u) \p_{\alpha}\varphi (x,t) \right)dx dt\geq
      \\
      -\intRR \left(|\nabla \varphi| + |\p_t \varphi|\right)
      d\mu(x,t) - \intR \varphi (x,0)d\mu_0(x) \\
       -\intRR \left(|\nabla \left[\varphi D\eta(u)\right]| + |\p_t
      \left[\varphi D\eta(u)\right]|\right) d\mub (x,t)
        \\
       - \intR [\varphi{D}\eta(u)](x,0) d\mubo (x) + \intRR
      \varphi \div u^T Z_{\alpha}(u^h,u) dx dt,
    \end{multline}
  where $Z_{\alpha}(u^h,u)= D^2 \eta(u)
  (f_{\alpha}(u^h)-f_{\alpha}(u)- \left(D f_{\alpha}(u)\right) (u^h-u))$.
\end{Prop}

\begin{proof}
  Let $\varphi$ be any nonnegative Lipschitz continuous test function
  with compact support in $\R^d \times [0,T]$.
  Since $u$ is a classical solution of
  \eqref{eq:P}--\eqref{eq:u0}, it satisfies
  \begin{equation*}
    \intRR \eta(u)\p_t \varphi(x,t)
    + \suma \xi_{\alpha}(u) \p_{\alpha}\varphi(x,t) dx dt
    +\intR\eta(u_0)\varphi(x,0)dx=0.
  \end{equation*}
  Subtracting this identity to \eqref{eq:ineq_entropy} yields
  \begin{multline}
    \label{eq:etauh-etau}
      \intRR (\eta(u^h)-\eta(u))\p_t \varphi(x,t) +
      \suma (\xi_\alpha(u^h)-\xi_\alpha(u)) \p_\alpha\varphi(x,t) dx
      dt \\
      - \geq \intRR (|\nabla \varphi| + |\p_t
      \varphi|) d\mu(x,t)-\intR\varphi(x,0)d\mu_0(x).
  \end{multline}
We now exhibit the relative entropy-relative entropy flux pair in
the inequality \eqref{eq:etauh-etau} and obtain
  \begin{multline}
  \label{eq:Huhu}
    \lefteqn{\intRR \left(H(u^h,u) \p_t \varphi +
    \suma Q_{\alpha}(u^h,u)\p_{\alpha}\varphi\right) dx dt 
    \geq}\\
   -\intRR |\nabla \varphi|+|\p_t \varphi|
   d\mu(x,t) -\intR \varphi(x,0)d\mu_0(x) \\
   - \intRR
   \left({D}\eta(u) \right)^T \left((u^h-u)\p_t \varphi +
   \suma (f_{\alpha}(u^h)-f_{\alpha}(u))
   \p_{\alpha}\varphi\right) dx dt.
  \end{multline}
Since $u$ is a strong solution of
\eqref{eq:P}--\eqref{eq:u0}, it satisfies the following weak
identity, $\forall \psi \in C_c(\R^d\times \R^+;\R^m)$
  \begin{equation}
  \label{eq:egentrop_u}
    \intRR \left[u \p_t\psi (x,t)+ \suma
    f_{\alpha}(u) \p_{\alpha} \psi (x,t)\right] dx dt + \intR u_0(x)
    \psi(x,0) dx =0.
\end{equation}
Then we combine Proposition~\ref{lem:laxwendroff} with \eqref{eq:egentrop_u}, 
so that 
using the Lipschitz continuous vector field $[\varphi{D}\eta(u)]$
as test function leads to
  \begin{multline}
  \label{eq:uh-u}
    -\intRR \left(D\eta(u)\right)^T  \left[(u^h-u)\p_t \varphi + \suma
    (f_{\alpha}(u^h)-f_{\alpha}(u)) \p_{\alpha}\varphi\right] dx dt \geq \\
    -\intRR |\nabla [\varphi D\eta(u)]| + |\p_t
    [\varphi D\eta(u)]| d\mub(x,t)
    - \intR [\varphi D\eta(u)](x,0) d\mubo(x) \\
    + \intRR (u^h-u)\varphi\p_t({D}\eta(u)) +
    \suma (f_{\alpha}(u^h)-f_{\alpha}(u)) \varphi
    \p_{\alpha}(D\eta(u)) dx dt.
\end{multline}
Moreover identity \eqref{eq:entro_commute} together with~\eqref{eq:P} gives
\begin{multline}
  \label{eq:nablaeta}
    \p_t(D\eta(u))  = \p_t u^T D^2 \eta(u) 
     = -\suma\p_{\alpha} \left(f_{\alpha}(u)\right)^T D^2 \eta(u) \\
     = -\suma \p_{\alpha} u^T Df_{\alpha}(u)^T  D^2\eta(u) 
    = -\suma \p_{\alpha} u^T D^2\eta(u) Df_{\alpha}(u).
\end{multline}
Injecting \eqref{eq:uh-u} and \eqref{eq:nablaeta} into \eqref{eq:Huhu}
leads to the conclusion.
\end{proof}

\begin{Lem}\label{lem:Z}
There exists $C_Z$ depending only on $f, \eta$ and $\O$ such that, for all 
$\alpha \in \{1,\dots, d\}$, 
\begin{equation}\label{eq:Z_quad}
|Z_\alpha(u^h,u)| \le C_Z |u^h - u|^2.
\end{equation}
\end{Lem}
\begin{proof}
For ${ M} : \O \to \R^{m\times m}$ and $\Upsilon:\O \to {\mathcal L}(\R^m; \R^{m\times m})$, we set
\begin{equation*}
  \left\|\, { M} \, \right\|_{\infty,\infty} = \sup_{u \in \O} \left|
    { M}(u)\right|_\infty, \qquad
  \left\|\, \Upsilon \, \right\|_{\infty,2} = \sup_{u \in \O}\left(
    \sup_{v \in \R^m, |v| = 1} \left| {\Upsilon}(u) \cdot
      v\right|_2\right),
\end{equation*}
where $|\cdot|_2$ and $|\cdot|_\infty$ denote the usual matrix $2$- and 
$\infty$-norms respectively.
Using the Taylor expansion of $f_\alpha$ around $u$, we get that 
$$
\left|f_\alpha(u^h) - f_\alpha(u) - \left(D f_\alpha(u)\right)(u^h - u) \right|
\le \frac12 \left\|\, D^2 f_\alpha \, \right\|_{\infty,2} |u^h - u|^2, 
$$
then, estimate~\eqref{eq:Z_quad} holds for 
$
C_Z = \frac12 \left\|\, D^2\eta \,
\right\|_{\infty,\infty}\left\|\, D^2 f_\alpha \,
\right\|_{\infty,2}.
$
\end{proof}

We now prove the following lemma on the finite speed of propagation.
\begin{Lem}\label{lem:s}
  Let $L_f$ be defined by~\eqref{eq:Lf_Rayleigh}, then, for all $s \ge L_f$, one has 
\begin{equation}\label{eq:finite_speed}
  sH(u^h,u) + \sum_{\alpha = 1}^d \frac{x_\alpha}{{|x|}} Q_\alpha(u^h,u) \ge 0. 
  \end{equation}
\end{Lem}

\begin{proof}
  Denote by $w^h := u^h -u$, then 
  it follows from the characterization~\eqref{eq:H_int} of the relative entropy $H$ that 
  \begin{equation}
    H =\int_0^1 \int_0^\theta {(w^h)}^T D^2_u \eta(u+\gamma w^h)
    w^h d\gamma d\theta. \label{eq:H1}
\end{equation}
Denoting by $\mathbb A_\gamma$ the symmetric definite positive matrix
$D^2_u \eta(u+\gamma w^h)$, 
and by $\langle \cdot \, , \, \cdot \rangle_{\mathbb A_\gamma}$ the
scalar product on $\R^n$ defined by $\langle v_1  , v_2 \rangle_{\mathbb A_\gamma} = v_1^T\mathbb A_\gamma v_2$, 
the relation~\eqref{eq:H1} can be rewritten 
\begin{equation}\label{eq:H2}
H = \int_0^1 \int_0^\theta \langle w^h  , w^h\rangle_{\mathbb
  A_\gamma} d\gamma d\theta.
\end{equation}
On the other hand, it follows from the definition~\eqref{eq:xi} of the
entropy flux $\xi$ that
\begin{align}
  Q_\alpha 	= & \int_0^1  \left(D\eta(u + \theta w^h) -
    D\eta(u) \right) \left(D f_\alpha(u + \theta
    w^h)\right)^T w^h d\theta \nn \\
  = & \int_0^1 \int_0^\theta \langle w^h, \left(D
    f_\alpha(u + \theta w^h)\right)^T w^h
  \rangle_{\mathbb A_\gamma} d \gamma d \theta \nn
\end{align}
for all $\alpha \in \{1,\dots, d\}$.
The quantity $L_f$ introduced in~\eqref{eq:Lf_Rayleigh} has been designed so that 
$
\left| \langle w^h, \left(D f_\alpha(u + \theta w^h)\right)^T
  w^h \rangle_{\mathbb A_\gamma}\right| \le L_f  \langle w^h, w^h
\rangle_{\mathbb A_\gamma}.
$
Therefore, we obtain  
\begin{equation}
| Q_\alpha | \le
 L_f  \int_0^1 \int_0^\theta  \langle w^h, w^h \rangle_{\mathbb
  A_\gamma} d \gamma d \theta = L_f H. \label{eq:majo_Q}
  \end{equation}
The fact that~\eqref{eq:finite_speed} holds is a straightforward
consequence of~\eqref{eq:majo_Q}.
\end{proof}

\subsection{End of the proof of Theorem~\ref{thm:err}}

We now have at hand all the tools needed for comparing $u^h$ to $u$ {\em via} 
the relative entropy $H(u^h,u)$. 

Let $\delta \in (0,T)$ be a parameter to be fixed later on, and, for $k \in \N$, we 
define the nonincreasing Lipschitz continuous function $\theta_{k} : \R^+ \to [0,1]$ by 
$$
\theta_{k}(t) = \min\left(1,\max\left(0, \frac{(k+1)\delta - t }{\delta}\right)\right), \quad \forall t \ge 0.
$$
Let us also introduce the Lipschitz continuous function $\psi: \R^d \times \R_+ \to [0,1]$ defined by 
$
\psi(x,t) = 1-\min\left(1 , \max\left(0 , {|x|} - r - L_f(T-t) +1 \right)\right), 
$
where $L_f$ is defined by~\eqref{eq:Lf_Rayleigh}.
The function $\varphi_k:(x,t)\in \R^d \times \R^+ \mapsto \theta_k(t) \psi(x,t) \in [0,1]$ 
can be considered as a test function in~\eqref{eq:ineq_relative_entrop}. Indeed, denoting by 
$$
{\Ii}_k^\delta = [k\delta, (k+1)\delta], \quad \Cc_{r,T}(t)  =
\left\{(x,t) \; | \; {|x|} \in [r+L_f(T-t),r+L_f(T-t)+1] \right\},  
$$
one has 
\begin{align*}
\p_t \varphi_k(x,t) =& - \frac1\delta {\bf 1}_{\Ii_k^\delta}(t) \psi(x,t) - 
L_f \theta_k(t) {\bf 1}_{\Cc_{r,T}(t)}(x), \\
\nabla \varphi_k(x,t) =& -\frac{x}{{|x|}} \theta_k(t) {\bf 1}_{\Cc_{r,T}(t)}(x), 
\end{align*}
so that both $\p_t \varphi_k$ and $|\nabla \varphi_k|$ belong to the set $E$ defined in 
Remark~\ref{rem:mu_mub}.
Then taking $\varphi_k$ as test function in~\eqref{eq:ineq_relative_entrop}  yields 
\begin{align}
&\frac{1}\delta \int_{\Ii_k^\delta} \int_{\R^d} H \psi dxdt +
\int_0^T \theta_k(t) \int_{\R^d} \sum_{\alpha = 1}^d \p_\alpha u^T
Z_\alpha (u^h,u)\; \psi  dxdt  \nn \\
&\qquad \le - \int_0^T \theta_k(t) \int_{\Cc_{r,T}(t)} \left(L_f H+
  \sum_{\alpha = 1}^d Q_\alpha \frac{x_\alpha}{{|x|}} \right) dxdt +
R_1 + R_2 + R_3 + R_4, \nn
\end{align}
where 
\begin{align*}
  R_1 =	& \iint_{\R^d \times [0,T]} (|\nabla \varphi_k(x,t)| +  |\p_t
  \varphi_k(x,t)|)  d \mu(x,t), \\
  R_2 =	&  \int_{\R^d} \psi(x,0) d\mu_0(x),\\
  R_3 =	&  \iint_{\R^d \times [0,T]}  |D\eta(u)| (|\nabla
  \varphi_k(x,t)| +  |\p_t \varphi_k(x,t)|) d \mub(x,t)\\
  		&  + \iint_{\R^d \times [0,T]} \varphi_k(x,t) {|D^2
                  \eta(u)(x,t) |}_\infty \left( | \p_t u | + | \nabla
                  u | \right) d \mub(x,t),\\
  R_4 =	&   \int_{\R^d} \psi(x,0) |D \eta(u_0)| d\mub_0(x). 
\end{align*}
Thanks to Lemma~\ref{lem:s}, one has
\be
\frac{1}\delta \int_{\Ii_k^\delta} \int_{\R^d} H \psi dxdt +  \int_0^T
\theta_k(t) \int_{\R^d} \sum_{\alpha = 1}^d \p_\alpha u^T Z_\alpha
(u^h,u)\; \psi  dxdt \le   R_1 + R_2 + R_3 + R_4. \label{eq:R1234}
\end{equation}
The definition of $\varphi_k$ ensures that ${\| \varphi_k \|}_\infty = 1$, 
${\|\nabla \varphi_k \|}_\infty\le 1$,  ${\| \p_t \varphi_k \|}_\infty \le \frac1\delta + L_f$ and 
$$
{\rm supp}(\varphi_k) \subset \bigcup_{t \in [0,(k+1)\delta]} B(0,r+L_f(T-t)+1) \times \{t\},
$$
This leads to 
$$
R_1 \le \left(\frac1\delta + L_f + 1\right) \mu({\rm supp}(\nabla
\varphi_k) \cup {\rm supp}(\p_t \varphi_k)).
$$
Thanks to Lemma~\ref{lem:borne_mu_mu0}, we obtain that there exists
$C_\mu^k$ (depending on $k$, $r$, $T$,  $\delta$, $L_f$, $a$, 
$\lambda^\star$, $u_0$, $G_{KL}$ and $\eta$) such that 
\begin{equation}\label{eq:R1}
R_1 \le C_\mu^k \left(\frac1\delta + L_f + 1\right) {\sqrt{h}}. 
\end{equation}
It follows from similar arguments that there exists $C_{\mu_0}^k$
(depending on $k$ $\eta$, $u_0$, $r$, $L_f$, $T$ and $\delta$) such
that 
\be \label{eq:R2}
R_2 \le  C_{\mu_0}^k h, 
\end{equation}
and, thanks to Lemma~\ref{lem:bornes_mub_mub0}, we obtain that there
exists $C_{\mub_0}^k$ (depending on $k$, $r, u_0, L_f,T$ and $\delta$) 
such that 
\be \label{eq:R4}
R_4 \le  C_{\mub_0}^k {\| D \eta(u_0) \|}_\infty  h.
\end{equation}
Similarly, there exists $C_{\mub}^k$ (depending on $k$, $T,r,L_f,a,
\lambda^\star, u_0, G_{KL}$ and $\delta$) such that 
\begin{equation}\label{eq:R3}
R_3 \le  C_{\mub}^k \left( {\| D\eta (u) \|}_\infty\left(\frac1\delta
    + L_f + 1\right)  +  {\| D^2 \eta \|}_{\infty,\infty}
  (\| \p_t u \|_\infty + \| \nabla u \|_\infty )\right) {\sqrt{h}}.
\end{equation}
By using Lemma~\ref{lem:Z} and $0 \le \theta_k(t) \le 1$, we obtain 
\begin{multline}\label{eq:R5}
\int_0^T \theta_k(t) \int_{\R^d} \sum_{\alpha = 1}^d \p_\alpha u^T Z_\alpha (u^h,u)\; \psi  dxdt   \\
 \qquad  \ge - C_Z {\|\nabla u\|}_\infty\iint_{\R^d \times
  [0,(k+1)\delta]} | u^h(x,t) - u(x,t) |^2 \psi(x,t)dxdt.
\end{multline}
Since the entropy $\eta$ is supposed to be $\beta_0$-convex, we have 
\begin{equation}\label{eq:H_beta}
H(x,t) \ge \frac{\beta_0}2  | u^h(x,t) - u(x,t) |^2 .
\end{equation}
Putting~\eqref{eq:R1}--\eqref{eq:H_beta} together with~\eqref{eq:R1234} provides  
\begin{multline}\label{eq:err1}
\left(\frac{\beta_0}{2 \delta} - C_Z {\| \nabla u \|}_\infty\right)
\int_{\Ii_k^\delta} \int_{\R^d} | u^h - u |^2 \psi\;  dxdt   \\
 \hspace{2cm}
\le C_Z {\| \nabla u \|}_\infty \iint_{\R^d \times [0,k\delta]} | u^h
- u |^2 \psi\;  dxdt + C_{k} {\sqrt{h}}, 
\end{multline}
where (recall that $h \le1$)
\begin{align*}
C_{k} = &C_\mu^k \left(\frac1\delta + L_f + 1\right) + C_{\mu_0}^k + C_{\mub_0}^k {\| D \eta(u_0) \|}_\infty \\
&+ C_{\mub}^k \left( {\| D \eta (u) \|}_\infty\left(\frac1\delta + L_f + 1\right)  +  {\| D^2 \eta \|}_{\infty,\infty}
(\| \p_t u \|_\infty + \| \nabla u \|_\infty )\right).
\end{align*}
Choose now 
$\delta = \frac{T}{p^\star+1}$ with  $p^\star =
\min\left\{ p \in \N^\star \; | \; \frac{T}{p+1} \le \frac{\beta_0}{2C_Z
    {\| \nabla u \|}_\infty + 2} \right\}$ 
(note that neither $\delta$ nor $p^\star$ depend on $h$),  
so that~\eqref{eq:err1} becomes 
\begin{equation}\label{eq:err2}
e_k \le \omega \sum_{i=0}^{k-1} e_i + C_k {\sqrt{h}}, 
\end{equation}
where 
$e_k = \int_{\Ii_k^\delta} \int_{\R^d} | u^h - u |^2 \psi\;  dxdt$ and  
$\omega = C_Z {\| \nabla u \|}_\infty.$
Hence, a few algebraic calculations allow us to claim that 
\begin{equation}\label{eq:N-1}
 \sum_{k=0}^{p^\star} e_k \le  {\sqrt{h}} \sum_{k=0}^{p^\star}
 C_k \left((1+\omega)^{p^\star - k + 1} - \omega\right).
\end{equation}
Noticing that $\psi(x,t) = 1$ if $x \in B(0,r + L_f(T-t))$, 
and that $\psi(x,t) \ge 0$ for all $(x,t) \in \R^d \times (0,T)$, one
finally has 
\begin{equation}\label{eq:N}
\int_0^T \int_{B(0, r-st)} |u-u^h|^2 dxdt \le  \int_0^T \int_{\R^d}
|u-u^h|^2 \psi(x,t)  dxdt = 
\sum_{k=0}^{p^\star} e_k.
\end{equation}
We conclude the proof using~\eqref{eq:N-1} in \eqref{eq:N}.

\section{Conclusion}

We analyzed the convergence of first order finite volume schemes
entering the framework detailed in~\cite{bouchut_book} and summarized
in~\S\ref{sssec:fluxes}. In~\S\ref{ssec:WBV}, we derived a
so-called~\emph{weak-BV} estimate based on the quantification of the
numerical entropy dissipation.  This estimate is new in the case of
time-explicit finite volume schemes.  It allows to prove some error
estimate between a numerical solution and a strong solution of order
$h^{1/4}$ in the space-time $L^2$-norm. Let us also mention that one
could use the weak-$BV$ estimate to prove to convergence to entropy
measure-valued solutions, following \cite{DiPerna85} (see also
\cite{HM14}).  On the other hand, strong solutions are global if one
adds some suitable entropy-dissipating relaxation
term~\cite{HN03,Yong04}), and our work could be extended to this
situation without any major difficulty by mixing our result with the
one proposed in~\cite{JR06}.


\begin{thebibliography}{10}

\bibitem{BGL93}
{\sc C.~Bardos, F.~Golse, and C.~D. Levermore}, {\em Fluid dynamic limits of
  kinetic equations. {II}. {C}onvergence proofs for the {B}oltzmann equation},
  Comm. Pure Appl. Math., 46 (1993), pp.~667--753.

\bibitem{BTV09}
{\sc F.~Berthelin, A.~E. Tzavaras, and A.~Vasseur}, {\em From discrete velocity
  {B}oltzmann equations to gas dynamics before shocks}, J. Stat. Phys., 135
  (2009), pp.~153--173.

\bibitem{BV05}
{\sc F.~Berthelin and A.~Vasseur}, {\em From kinetic equations to
  multidimensional isentropic gas dynamics before shocks}, SIAM J. Math. Anal.,
  36 (2005), pp.~1807--1835.

\bibitem{BGP}
{\sc D.~Bouche, J.-M. Ghidaglia, and F.~P. Pascal}, {\em An optimal error
  estimate for upwind finite volume methods for nonlinear hyperbolic
  conservation laws}, Appl. Numer. Math., 61 (2011), pp.~1114--1131.

\bibitem{bouchut_book}
{\sc F.~Bouchut}, {\em Nonlinear stability of finite volume methods for
  hyperbolic conservation laws and well-balanced schemes for sources},
  Frontiers in Mathematics, Birkh{\"a}user Verlag, Basel, 2004.

\bibitem{BP98}
{\sc F.~Bouchut and B.~Perthame}, {\em Kru\v zkov's estimates for scalar
  conservation laws revisited}, Trans. Amer. Math. Soc., 350 (1998),
  pp.~2847--2870.

\bibitem{chainais99}
{\sc C.~Chainais-Hillairet}, {\em Finite volume schemes for a nonlinear
  hyperbolic equation. {C}onvergence towards the entropy solution and error
  estimate}, M2AN Math. Model. Numer. Anal., 33 (1999), pp.~129--156.

\bibitem{CH00}
\leavevmode\vrule height 2pt depth -1.6pt width 23pt, {\em Second-order
  finite-volume schemes for a non-linear hyperbolic equation: error estimate},
  Math. Methods Appl. Sci., 23 (2000), pp.~467--490.

\bibitem{chainais01}
{\sc C.~Chainais-Hillairet and S.~Champier}, {\em Finite volume schemes for
  nonhomogeneous scalar conservation laws: error estimate}, Numer. Math., 88
  (2001), pp.~607--639.

\bibitem{CGH93}
{\sc S.~Champier, T.~Gallou\"et, and R.~Herbin}, {\em Convergence of an
  upstream finite volume scheme for a nonlinear hyperbolic equation on a
  triangular mesh}, Numer. Math., 66 (1993), pp.~139--157.

\bibitem{CClF94}
{\sc B.~Cockburn, F.~Coquel, and P.~G. LeFloch}, {\em An error estimate for
  finite volume methods for multidimensional conservation laws}, Math. Comp.,
  63 (1994), pp.~77--103.

\bibitem{cgs}
{\sc F.~Coquel, E.~Godlewski, and N.~Seguin}, {\em Relaxation of fluid
  systems}, Math. Models Methods Appl. Sci., 22 (2012), pp.~1250014, 52.

\bibitem{coqperth}
{\sc F.~Coquel and B.~Perthame}, {\em Relaxation of energy and approximate
  {R}iemann solvers for general pressure laws in fluid dynamics}, SIAM J.
  Numer. Anal., 35 (1998), pp.~2223--2249 (electronic).

\bibitem{Daf79}
{\sc C.~M. Dafermos}, {\em The second law of thermodynamics and stability},
  Arch. Rational Mech. Anal., 70 (1979), pp.~167--179.

\bibitem{dafermosBook}
\leavevmode\vrule height 2pt depth -1.6pt width 23pt, {\em Hyperbolic
  conservation laws in continuum physics}, vol.~325 of Grundlehren der
  Mathematischen Wissenschaften [Fundamental Principles of Mathematical
  Sciences], Springer-Verlag, Berlin, third~ed., 2010.

\bibitem{DLS09}
{\sc C.~De~Lellis and L.~Sz{\'e}kelyhidi, Jr.}, {\em The {E}uler equations as a
  differential inclusion}, Ann. of Math. (2), 170 (2009), pp.~1417--1436.

\bibitem{DLS10}
\leavevmode\vrule height 2pt depth -1.6pt width 23pt, {\em On admissibility
  criteria for weak solutions of the {E}uler equations}, Arch. Ration. Mech.
  Anal., 195 (2010), pp.~225--260.

\bibitem{DL11}
{\sc F.~Delarue and F.~Lagouti{\`e}re}, {\em Probabilistic analysis of the
  upwind scheme for transport equations}, Arch. Ration. Mech. Anal., 199
  (2011), pp.~229--268.

\bibitem{Des04}
{\sc B.~Despr{\'e}s}, {\em An explicit a priori estimate for a finite volume
  approximation of linear advection on non-{C}artesian grids}, SIAM J. Numer.
  Anal., 42 (2004), pp.~484--504.

\bibitem{DiP79}
{\sc R.~J. DiPerna}, {\em Uniqueness of solutions to hyperbolic conservation
  laws}, Indiana Univ. Math. J., 28 (1979), pp.~137--188.

\bibitem{DiPerna85}
\leavevmode\vrule height 2pt depth -1.6pt width 23pt, {\em Measure-valued
  solutions to conservation laws}, Arch. Rational Mech. Anal., 88 (1985),
  pp.~223--270.

\bibitem{EGGH98}
{\sc R.~Eymard, T.~Gallou{\"e}t, M.~Ghilani, and R.~Herbin}, {\em Error
  estimates for the approximate solutions of a nonlinear hyperbolic equation
  given by finite volume scheme}, IMA J. Numer. Anal., 18 (1998), pp.~563--594.

\bibitem{FN12}
{\sc E.~Feireisl and A.~Novotn{\'y}}, {\em Weak-strong uniqueness property for
  the full {N}avier-{S}tokes-{F}ourier system}, Arch. Ration. Mech. Anal., 204
  (2012), pp.~683--706.

\bibitem{Fri54}
{\sc K.~O. Friedrichs}, {\em Symmetric hyperbolic linear differential
  equations}, Comm. Pure Appl. Math., 7 (1954), pp.~345--392.

\bibitem{GR96}
{\sc E.~Godlewski and P.-A. Raviart}, {\em Numerical approximation of
  hyperbolic systems of conservation laws}, vol.~118 of Applied Mathematical
  Sciences, Springer-Verlag, New York, 1996.

\bibitem{godunov}
{\sc S.~K. Godunov}, {\em A difference method for numerical calculation of
  discontinuous solutions of the equations of hydrodynamics}, Mat. Sb. (N.S.),
  47 (89) (1959), pp.~271--306.

\bibitem{HN03}
{\sc B.~Hanouzet and R.~Natalini}, {\em Global existence of smooth solutions
  for partially dissipative hyperbolic systems with a convex entropy}, Arch.
  Ration. Mech. Anal., 169 (2003), pp.~89--117.

\bibitem{HLL}
{\sc A.~Harten, P.~D. Lax, and B.~van Leer}, {\em On upstream differencing and
  {G}odunov-type schemes for hyperbolic conservation laws}, SIAM Rev., 25
  (1983), pp.~35--61.

\bibitem{HM14}
{\sc A.~Hiltebrand and Si. Mishra}, {\em Entropy stable shock capturing
  space-time discontinuous {G}alerkin schemes for systems of conservation
  laws}, Numer. Math., 126 (2014), pp.~103--151.

\bibitem{JP}
{\sc C.~Johnson and J.~Pitk{\"a}ranta}, {\em An analysis of the discontinuous
  {G}alerkin method for a scalar hyperbolic equation}, Math. Comp., 46 (1986),
  pp.~1--26.

\bibitem{JR05}
{\sc V.~Jovanovi{\'c} and C.~Rohde}, {\em Finite-volume schemes for
  {F}riedrichs systems in multiple space dimensions: a priori and a posteriori
  error estimates}, Numer. Methods Partial Differential Equations, 21 (2005),
  pp.~104--131.

\bibitem{JR06}
\leavevmode\vrule height 2pt depth -1.6pt width 23pt, {\em Error estimates for
  finite volume approximations of classical solutions for nonlinear systems of
  hyperbolic balance laws}, SIAM J. Numer. Anal., 43 (2006), pp.~2423--2449.

\bibitem{Kato75}
{\sc T.~Kato}, {\em The {C}auchy problem for quasi-linear symmetric hyperbolic
  systems}, Arch. Rational Mech. Anal., 58 (1975), pp.~181--205.

\bibitem{ohlberger00}
{\sc D.~Kr{\"o}ner and M.~Ohlberger}, {\em A posteriori error estimates for
  upwind finite volume schemes}, Math. Comp., 69 (2000), pp.~25--39.

\bibitem{Kru70}
{\sc S.~N. Kruzhkov}, {\em First order quasilinear equations with several
  independent variables.}, Mat. Sb. (N.S.), 81 (1970), pp.~228--255.

\bibitem{Kuz76}
{\sc N.~N. Kuznetsov}, {\em The accuracy of certain approximate methods for the
  computation of weak solutions of a first order quasilinear equation}, \v Z.
  Vy\v cisl. Mat. i Mat. Fiz., 16 (1976), pp.~1489--1502, 1627.

\bibitem{Lax2}
{\sc P.~D. Lax}, {\em Hyperbolic systems of conservation laws. {II}}, Comm.
  Pure Appl. Math., 10 (1957), pp.~537--566.

\bibitem{LV11}
{\sc N.~Leger and A.~Vasseur}, {\em Relative entropy and the stability of
  shocks and contact discontinuities for systems of conservation laws with
  non-{BV} perturbations}, Arch. Ration. Mech. Anal., 201 (2011), pp.~271--302.

\bibitem{LiTaTsien_book}
{\sc T.~T. Li}, {\em Global classical solutions for quasilinear hyperbolic
  systems}, vol.~32 of RAM: Research in Applied Mathematics, Masson, Paris;
  John Wiley \& Sons, Ltd., Chichester, 1994.

\bibitem{lions_book}
{\sc P.-L. Lions}, {\em Mathematical topics in fluid mechanics. {V}ol. 1-2},
  vol.~10 of Oxford Lecture Series in Mathematics and its Applications, The
  Clarendon Press, Oxford University Press, New York, 1996--1998.
\newblock Compressible models, Oxford Science Publications.

\bibitem{MV07}
{\sc B.~Merlet and J.~Vovelle}, {\em Error estimate for finite volume scheme},
  Numer. Math., 106 (2007), pp.~129--155.

\bibitem{rusanov}
{\sc V.~V. Rusanov}, {\em Calculation of interaction of non-steady shock waves
  with obstacles}, J. Comp. Math. Phys. USSR, 1 (1961), pp.~267--279.

\bibitem{StR09}
{\sc L.~Saint-Raymond}, {\em Hydrodynamic limits: some improvements of the
  relative entropy method}, Ann. Inst. H. Poincar\'e Anal. Non Lin\'eaire, 26
  (2009), pp.~705--744.

\bibitem{SzCV}
{\sc A.~Szepessy}, {\em Convergence of a streamline diffusion finite element
  method for scalar conservation laws with boundary conditions}, RAIRO Mod\'el.
  Math. Anal. Num\'er., 25 (1991), pp.~749--782.

\bibitem{Tadmor87}
{\sc E.~Tadmor}, {\em The numerical viscosity of entropy stable schemes for
  systems of conservation laws. {I}}, Math. Comp., 49 (1987), pp.~91--103.

\bibitem{TadmorBook}
\leavevmode\vrule height 2pt depth -1.6pt width 23pt, {\em Entropy stability
  theory for difference approximations of nonlinear conservation laws and
  related time-dependent problems}, Acta Numer., 12 (2003), pp.~451--512.

\bibitem{Tza05}
{\sc A.~E. Tzavaras}, {\em Relative entropy in hyperbolic relaxation}, Commun.
  Math. Sci., 3 (2005), pp.~119--132.

\bibitem{vL79}
{\sc B.~Van~Leer}, {\em Towards the ultimate conservative difference scheme. v.
  a second-order sequel to godunov's method}, J. Comput. Phys., 32 (1979),
  pp.~101--136.

\bibitem{vila94}
{\sc J.-P. Vila}, {\em Convergence and error estimates in finite volume schemes
  for general multidimensional scalar conservation laws. {I}. {E}xplicit
  monotone schemes}, RAIRO Modél. Math. Anal. Numér., 28 (1994), pp.~267--295.

\bibitem{Yau91}
{\sc H.-T. Yau}, {\em Relative entropy and hydrodynamics of {G}inzburg-{L}andau
  models}, Lett. Math. Phys., 22 (1991), pp.~63--80.

\bibitem{Yong04}
{\sc W.A. Yong}, {\em Entropy and global existence for hyperbolic balance
  laws}, Arch. Ration. Mech. Anal., 172 (2004), pp.~247--266.

\end{thebibliography}

\def\cprime{$'$}

\end{document}